\documentclass[a4paper, twoside,10pt]{article}

\usepackage{amsmath}
\usepackage{amssymb}
\usepackage{amsfonts}
\usepackage{graphicx}
\usepackage{subfigure}
\usepackage{eucal}
\usepackage{amsthm}
\usepackage{verbatim}
\usepackage{epsfig}
\usepackage{hyperref}

\newcommand{\brk}[1]{\left( #1 \right)}
\newcommand{\abrk}[1]{\left| #1 \right|}
\newcommand{\cbrk}[1]{\left\{ #1 \right\}}
\newcommand{\sbrk}[1]{\left[ #1 \right]}

\def\Rc{\mathbb{R}}

\newcommand{\F}{\mathcal{F}}

\newtheorem{thm}{Theorem}
\newtheorem{prop}{Proposition}
\newtheorem{lem}{Lemma} 
\newtheorem{cor}{Corollary}
\newtheorem{definition}{Definition}

\newtheorem{remark}{Remark}

\begin{document}

\title{Certain inequalities involving
prolate spheroidal wave functions and associated quantities}
\author{Andrei Osipov\footnote{This author's research was supported in part
 by the AFOSR grant \#FA9550-09-1-0241}
\footnote{Yale University, 51 Prospect st, New Haven, CT 06511.
Email: andrei.osipov@yale.edu.
}}
\maketitle

\begin{abstract}
Prolate spheroidal wave functions (PSWFs) 
play an important role in various areas,
from physics (e.g. wave phenomena, fluid dynamics) to 
engineering (e.g. signal processing, filter design). 
Even though the significance
of PSWFs was realized at least half a century ago, 
and they frequently occur in applications, their analytical
properties have not been investigated as much as those of many
other special functions. In particular,
despite some recent progress,
the gap
between asymptotic expansions and numerical experience,
on the one hand, and rigorously proven explicit bounds and
estimates, on the other hand, is still rather wide.

This paper attempts to improve the current situation.
We analyze the differential operator associated with PSWFs,
to derive fairly tight estimates on its eigenvalues. 
By combining these inequalities with a number
of standard techniques,
we also obtain several other properties of the PSFWs.
The results are illustrated via numerical experiments.
\end{abstract}


\noindent
{\bf Math subject classification:} {
33E10, 34L15, 35S30, 42C10}


\section{Introduction}
\label{sec_intro}
The principal purpose of this paper is to provide
proofs for several inequalities involving bandlimited
functions (see Section~\ref{sec_summary} below). While 
some of these inequalities are known from
``numerical experience'' (see, for example,
\cite{ProlateSlepian1},
\cite{ProlateLandau1},
\cite{ProlateLandau2},
\cite{SlepianAsymptotic}),
their proofs appear to be absent in the literature.

A function $f: \Rc \to \Rc$ is bandlimited of band limit $c>0$, if there
exists a function $\sigma \in L^2\left[-1,1\right]$ such that
\begin{align}
f(x) = \int_{-1}^1 \sigma(t) e^{icxt} \; dt.
\label{eq_intro_f}
\end{align}
In other words, the Fourier transform of a bandlimited function
is compactly supported.
While \eqref{eq_intro_f} defines $f$ for all real $x$, 
one is often interested in bandlimited functions, whose 
argument is confined to an interval, e.g. $-1 \leq x \leq 1$.
Such functions are encountered in physics (wave phenomena,
fluid dynamics), engineering (signal processing), etc.
(see e.g. \cite{SlepianComments}, \cite{Flammer}, \cite{Papoulis}).

About 50 years ago it was observed that the eigenfunctions of
the integral operator $F_c: L^2\left[-1,1\right] \to L^2\left[-1,1\right]$,
defined via the formula
\begin{align}
F_c\sbrk{\varphi} \brk{x} = \int_{-1}^1 \varphi(t) e^{icxt} \; dt,
\label{eq_intro_fc}
\end{align}
provide a natural tool for dealing with bandlimited functions, defined
on the interval $\left[-1,1\right]$. Moreover, it
was observed 
(see, for example,
\cite{ProlateSlepian1}, \cite{ProlateLandau1}, \cite{ProlateSlepian2})
that the eigenfunctions of $F_c$
are precisely the prolate spheroidal wave functions (PSWFs),
well known from the mathematical physics (see, for example,
\cite{PhysicsMorse}, \cite{Flammer}).
The PSWFs are the eigenfunctions
of the differential operator $L_c$, defined via the formula
\begin{align}
L_c\left[ \varphi \right] \left(x\right)= 
-\frac{d}{dx} \left( (1-x^2) \cdot \frac{d\varphi}{dx}(x) \right) +
c^2 x^2.
\label{eq_intro_lc}
\end{align}
In other words, the integral operator $F_c$ 
commutes with
the differential
operator $L_c$ (see
\cite{ProlateSlepian1}, \cite{Grunbaum}).
This property, being remarkable by itself,
also plays an important role in both the analysis of PSWFs
and the associated numerical algorithms (see, for example,
\cite{Glaser}, \cite{RokhlinXiaoProlate}).

It is perhaps surprising, however, that the analytical
properties of PSWFs have not been investigated as
thoroughly as those of several other classes of special functions.
In particular,
when one reads through the classical works about the 
PSWFs (see, for example,
\cite{ProlateSlepian1},
\cite{ProlateLandau1},
\cite{ProlateLandau2},
\cite{ProlateSlepian2},
\cite{ProlateSlepian3}),
one is amazed by the number of properties stated without rigorous
proofs. Some other properties are only supported by analysis
of an asymptotic nature; see, for example, 
\cite{RokhlinXiaoAsymptotic},
\cite{LandauWidom},
\cite{SlepianAsymptotic},
\cite{Fuchs}.
To some extent,
this problem has been
addressed in a number of recently published papers,
for example, \cite{Yoel}, \cite{RokhlinXiaoProlate}, 
\cite{RokhlinXiaoApprox}.
Still, the gap between numerical experience and asymptotic expansions,
on the one hand, and rigorously proven explicit bounds and estimates,
on the other hand, is rather wide; this paper
offers a partial remedy for this deficiency.

This paper is mostly devoted to the analysis
of the differential operator $L_c$, defined via \eqref{eq_intro_lc}.
In particular, several explicit bounds for the eigenvalues of $L_c$
are derived. These bounds turn out to be fairly tight, and the
resulting inequalities lead to rigorous proofs
of several other properties of PSWFs. The analysis 
is illustrated through several numerical experiments.

The analysis of the eigenvalues of the integral operator $F_c$,
defined via \eqref{eq_intro_fc}, requires tools
different from those used in this paper; it will be published at a later date.
The implications of the analysis of both $L_c$ and $F_c$
to numerical algorithms involving
PSWFs are being currently investigated.

This paper is organized as follows. In Section~\ref{sec_prel},
we summarize a number of well known mathematical facts
to be used in the rest of this paper. In Section~\ref{sec_summary},
we provide a summary of the principal results
of this paper. In Section~\ref{sec_analytical}, we introduce
the necessary analytical apparatus and carry out the analysis.
In Section~\ref{sec_numerical}, we illustrate the analysis
via several numerical examples.

\section{Mathematical and Numerical Preliminaries}
\label{sec_prel}
In this section, we introduce notation and summarize
several facts to be used in the rest of the paper.

\subsection{Prolate Spheroidal Wave Functions}
\label{sec_pswf}
In this subsection, we summarize several facts about
the PSWFs. Unless stated otherwise, all these facts can be 
found in \cite{RokhlinXiaoProlate}, 
\cite{RokhlinXiaoApprox},
\cite{LandauWidom},
\cite{ProlateSlepian1},
\cite{ProlateLandau1}.

Given a real number $c > 0$, we define the operator
$F_c: L^2\sbrk{-1, 1} \to L^2\sbrk{-1, 1}$ via the formula
\begin{align}
F_c\sbrk{\varphi} \brk{x} = \int_{-1}^1 \varphi(t) e^{icxt} \; dt.
\label{eq_pswf_fc}
\end{align}
Obviously, $F_c$ is compact. We denote its eigenvalues by
$\lambda_0, \lambda_1, \dots, \lambda_n, \dots$ and assume that
they are ordered such that $\abrk{\lambda_n} \geq \abrk{\lambda_{n+1}}$
for all natural $n \geq 0$. We denote by $\psi_n$ the eigenfunction
corresponding to $\lambda_n$. In other words, the following
identity holds for all integer $n \geq 0$ and all real $-1 \leq x \leq 1$:
\begin{align}
\label{eq_prolate_integral}
\lambda_n \psi_n\brk{x} = \int_{-1}^1 \psi_n(t) e^{icxt} \; dt.
\end{align}
We adopt the convention\footnote{
This convention agrees with that of \cite{RokhlinXiaoProlate},
\cite{RokhlinXiaoApprox} and differs from that of \cite{ProlateSlepian1}.
}
that $\| \psi_n \|_{L^2\sbrk{-1,1}} = 1$.
The following theorem describes the eigenvalues and eigenfunctions
of $F_c$
(see
\cite{RokhlinXiaoProlate},
\cite{RokhlinXiaoApprox},
\cite{ProlateSlepian1}).
\begin{thm}
Suppose that $c>0$ is a real number, and that the operator $F_c$
is defined via \eqref{eq_pswf_fc} above. Then,
the eigenfunctions $\psi_0, \psi_1, \dots$ of $F_c$ are purely real,
are orthonormal and are complete in $L^2\sbrk{-1, 1}$.
The even-numbered functions are even, the odd-numbered ones are odd.
Each function $\psi_n$ has exactly $n$ simple roots in $\brk{-1, 1}$.
All eigenvalues $\lambda_n$ of $F_c$ are non-zero and simple;
the even-numbered ones are purely real and the odd-numbered ones
are purely imaginary; in particular, $\lambda_n = i^n \abrk{\lambda_n}$.
\label{thm_pswf_main}
\end{thm}
We define the self-adjoint operator
$Q_c: L^2\sbrk{-1, 1} \to L^2\sbrk{-1, 1}$ via the formula
\begin{align}
Q_c\sbrk{\varphi} \brk{x} =
\frac{1}{\pi} \int_{-1}^1 
\frac{ \sin \brk{c\brk{x-t}} }{x-t} \; \varphi(t) \; dt.
\label{eq_pswf_qc}
\end{align}
Clearly, if we denote by $\F:L^2(\Rc) \to L^2(\Rc)$ 
the unitary Fourier transform,
then
\begin{align}
Q_c\sbrk{\varphi} \brk{x} = 
\chi_{\sbrk{-1,1}}(x) \cdot
\F^{-1} \sbrk{ 
  \chi_{\sbrk{-c,c}}(\xi) \cdot \F\sbrk{\varphi}(\xi)
}(x),
\end{align}
where $\chi_{\left[-a,a\right]} : \Rc \to \Rc$ is the characteristic
function of the interval $\left[-a,a\right]$, defined via the formula
\begin{align}
\chi_{\left[-a,a\right]}(x) = 
\begin{cases}
1 & -a \leq x \leq a, \\
0 & \text{otherwise},
\end{cases}
\label{eq_char_function}
\end{align}
for all real $x$.
In other words, $Q_c$ represents low-passing followed by time-limiting.
$Q_c$ relates to $F_c$, defined via \eqref{eq_pswf_fc}, by 
\begin{align}
Q_c = \frac{ c }{ 2 \pi } \cdot F_c^{\ast} \cdot F_c,
\label{eq_pswf_qc_fc}
\end{align}
and the eigenvalues $\mu_n$ of $Q_n$ satisfy the identity
\begin{align}
\mu_n = \frac{c}{2\pi} \cdot \abrk{\lambda_n}^2,
\label{eq_prolate_mu}
\end{align}
for all integer $n \geq 0$.
Moreover, $Q_c$ has the same eigenfunctions $\psi_n$ as $F_c$.
In other words,
\begin{align}
\mu_n \psi_n(x) = \frac{1}{\pi} 
      \int_{-1}^1 \frac{ \sin\brk{c(x-t)} }{ x - t } \; \psi_n(t) \; dt,
\label{eq_prolate_integral2}
\end{align}
for all integer $n \geq 0$ and all $-1 \leq x \leq 1$.
Also,  $Q_c$ is closely related to the operator
$P_c: L^2(\Rc) \to L^2(\Rc)$,
defined via the formula
\begin{align}
P_c\sbrk{\varphi} \brk{x} =
\frac{1}{\pi} \int_{-\infty}^{\infty}
\frac{ \sin \brk{c\brk{x-t}} }{x-t} \; \varphi(t) \; dt,
\label{eq_pswf_pc}
\end{align}
which is a widely known orthogonal projection onto the space
of functions of band limit $c > 0$ on the real
line $\Rc$.

The following theorem about the eigenvalues $\mu_n$ of the operator $Q_c$,
defined via
\eqref{eq_pswf_qc},
can be traced back to \cite{LandauWidom}:
\begin{thm}
Suppose that $c>0$ and $0<\alpha<1$ are positive real numbers,
and that the operator $Q_c: L^2\left[-1,1\right] \to L^2\left[-1,1\right]$
is defined via \eqref{eq_pswf_qc} above.
Suppose also that the integer $N(c,\alpha)$ is the number of 
the eigenvalues $\mu_n$ of $Q_c$ that are greater than $\alpha$. In
other words,
\begin{align}
N(c,\alpha) = \max\left\{ k = 1,2,\dots \; : \; \mu_{k-1} > 0\right\}.
\end{align}
Then,
\begin{align}
N(c,\alpha) =
\frac{2}{\pi} c + \brk{ \frac{1}{\pi^2} \log \frac{1-\alpha}{\alpha} }
    \log c + O\brk{ \log c }.
\label{eq_mu_spectrum}
\end{align}
\label{thm_mu_spectrum}
\end{thm}
According to \eqref{eq_mu_spectrum}, there are about $2c/\pi$
eigenvalues whose absolute value is close to one, order of $\log c$
eigenvalues that decay exponentially, and the rest of them are
very close to zero. 

The eigenfunctions $\psi_n$ of $Q_c$ turn out to be the PSWFs, well
known from classical mathematical physics 
(see, for example,
\cite{PhysicsMorse}, \cite{Flammer}).
The following theorem, proved in a more general form in
\cite{ProlateSlepian2},
formalizes this statement.
\begin{thm}
For any $c > 0$, there exists a strictly increasing unbounded sequence
of positive numbers $\chi_0 <  \chi_1 <  \dots$ such that, for
each integer $n \geq 0$, the differential equation
\begin{align}
\brk{1 - x^2} \psi''(x) - 2 x \cdot \psi'(x) 
+ \brk{\chi_n - c^2 x^2} \psi(x) = 0
\label{eq_prolate_ode}
\end{align}
has a solution that is continuous on $\sbrk{-1, 1}$.
Moreover, all such solutions are constant multiples of 
the eigenfunction $\psi_n$ of $F_c$,
defined via \eqref{eq_pswf_fc} above.
\label{thm_prolate_ode}
\end{thm}
\noindent
The following theorem provides lower and upper bounds on $\chi_n$
of Theorem~\ref{thm_prolate_ode} (see, for example,
\cite{RokhlinXiaoProlate},
\cite{ProlateSlepian1},
\cite{ProlateLandau1}).
\begin{thm}
For all real $c>0$ and all natural $n \geq 0$,
\begin{align}
\label{eq_khi_crude}
n \brk{n + 1} < \chi_n < n \brk{n + 1} + c^2.
\end{align}
\label{thm_khi_crude}
\end{thm}
\noindent
The following result 
provides an upper bound on $\psi_n^2(1)$ (see \cite{RokhlinXiaoApprox}).
\begin{thm}
For all $c > 0$ and all natural $n \geq 0$,
\begin{align}
\psi_n^2(1) < n + \frac{1}{2}.
\end{align}
\label{thm_psi1_upper_bound}
\end{thm}

\subsection{Elliptic Integrals}
\label{sec_elliptic}
In this subsection, we summarize several facts about
elliptic integrals. These facts can be found, for example,
in section 8.1 in \cite{Ryzhik}, and in \cite{Abramovitz}.

The incomplete elliptic integrals of the first and second kind
are defined, respectively, by the formulae
\begin{align}
\label{eq_F_y}
& F(y, k) =  \int_0^y \frac{dt}{\sqrt{1 - k^2 \sin^2 t}}, \\
& E(y, k) = \int_0^y \sqrt{1 - k^2 \sin^2 t} \; dt,
\label{eq_E_y}
\end{align}
where $0 \leq y \leq \pi/2$ and $0 \leq k \leq 1$.
By performing the substitution $x = \sin t$, we can write 
\eqref{eq_F_y} and \eqref{eq_E_y} as
\begin{align}
& F(y, k) = \int_0^{\sin(y)}
  \frac{ dx }{ \sqrt{\brk{1 - x^2} \brk{1 - k^2 x^2} } },
\label{eq_F_y_2} \\
\nonumber \\
& E(y, k) = \int_0^{\sin(y)}
\sqrt{ \frac{1 - k^2 x^2}{1 - x^2} } \; dx.
\label{eq_E_y_2}
\end{align}
The complete elliptic integrals of the first and second kind are
defined, respectively, by the formulae
\begin{align}
\label{eq_F}
& F(k) = F\brk{\frac{\pi}{2}, k} = 
\int_0^{\pi/2} \frac{dt}{\sqrt{1 - k^2 \sin^2 t}}, \\
& E(k) = E\brk{\frac{\pi}{2}, k} =
\int_0^{\pi/2} \sqrt{1 - k^2 \sin^2 t} \; dt,
\label{eq_E}
\end{align}
for all $0 \leq k \leq 1$. Moreover,
\begin{align}
E\left( \sqrt{1-k^2} \right) =
1 + \left(-\frac{1}{4}+\log(2)-\frac{\log(k)}{2}\right) \cdot k^2 +
    O\left( k^4 \cdot \log(k) \right).
\label{eq_E_exp}
\end{align}

\subsection{Oscillation Properties of Second Order ODEs}
\label{sec_oscillation_ode}
In this subsection, we state several well known facts from the general theory
of second order ordinary differential equations (see e.g. \cite{Miller}).

The following two theorems appear in Section 3.6 of \cite{Miller}
in a slightly different form.
\begin{thm}[distance between roots]
Suppose that $h(t)$ is a solution of the ODE
\begin{align}
y''(t) + Q(t) y(t) = 0.
\label{eq_25_09_ode}
\end{align}
Suppose also that $x < y$ are two consecutive roots of $h(t)$, and
that
\begin{align}
A^2 \leq Q(t) \leq B^2, 
\label{eq_25_09_cond}
\end{align}
for all $x \leq t \leq y$.
Then,
\begin{align}
\frac{\pi}{B} < y - x < \frac{\pi}{A}.
\label{eq_25_09_thm}
\end{align}
\label{thm_25_09}
\end{thm}
\begin{thm}
Suppose that $a<b$ are real numbers, and that $g:(a,b) \to \Rc$
is a continuous monotone function.
Suppose also that
$y(t)$ is a solution of the ODE
\begin{align}
 y''(t) + g(t) \cdot y(t) = 0,
\label{eq_25_09_ode2}
\end{align}
in the interval $(a,b)$.
Suppose furthermore that
\begin{align}
t_1 < t_2 < t_3 < \dots
\end{align}
are consecutive roots of $y(t)$. 
If $g$ is non-decreasing, then
\begin{align}
t_2 - t_1 \geq t_3 - t_2 \geq t_4 - t_3 \geq \dots.
\label{eq_shrink}
\end{align}
If $g$ is non-increasing, then
\begin{align}
t_2-t_1 \leq t_3-t_2 \leq t_4-t_3 \leq \dots.
\label{eq_stretch}
\end{align}
\label{thm_25_09_2}
\end{thm}
\noindent
The following theorem is a special case
of Theorem 6.2 from Section 3.6 in \cite{Miller}:
\begin{thm}
Suppose that $g_1, g_2$ are continuous
functions, and that, for all real $t$ in the interval $\brk{a, b}$,
the inequality $g_1(t) < g_2(t)$ holds. 
Suppose also that $\phi_1, \phi_2:(a,b) \to \Rc$ are real-valued functions,
and that
\begin{align}
& \phi_1''(t) + g_1(t) \cdot \phi_1(t) = 0, \nonumber \\
& \phi_2''(t) + g_2(t) \cdot \phi_2(t) = 0,
\label{eq_08_12_g12}
\end{align}
for all $a < t < b$.
Then, $\phi_2$ has a root 
between every two consecutive roots of $\phi_1$.
\label{thm_08_12_zeros}
\end{thm}
\begin{cor}
Suppose that the functions $\phi_1, \phi_2$
are those of Theorem~\ref{thm_08_12_zeros} above.
Suppose also that
\begin{align}
\phi_1(t_0) = \phi_2(t_0), \quad \phi_1'(t_0) = \phi_2'(t_0),
\end{align}
for some $a < t_0 < b$.
Then, $\phi_2$ has at least as many roots in $\brk{t_0, b}$ as $\phi_1$.
\label{cor_08_12_zeros}
\end{cor}
\begin{proof}
Due to Theorem~\ref{thm_08_12_zeros}, we only need to show
that if $t_1$ is the minimal root of $\phi_1$ in $\brk{t_0, b}$,
then there exists a root of $\phi_2$ in $\brk{t_0, t_1}$.
By contradiction, suppose that this is not the case. In addition, without
loss of generality, suppose that $\phi_1(t), \phi_2(t)$ are positive
in $\brk{t_0, t_1}$. Then, due to \eqref{eq_08_12_g12},
\begin{align}
\phi_1'' \phi_2 - \phi_2'' \phi_1 = \brk{g_2 - g_1} \phi_1 \phi_2,
\end{align}
and hence
\begin{align}
0 & \; <
\int_{t_0}^{t_1} \brk{g_2(s) - g_1(s)} \phi_1(s) \phi_2(s) ds 
\nonumber \\
& \; = 
\sbrk{\phi_1'(s) \phi_2(s) - \phi_1(s) \phi_2'(s)}_{t_0}^{t_1} \nonumber \\
& \; =
\phi_1'(t_1) \phi_2(t_1) \leq 0,
\end{align}
which is a contradiction.
\end{proof}

\subsection{Pr\"ufer Transformations}
\label{sec_prufer}

In this subsection, we describe the classical Pr\"ufer
transformation of a second order ODE 
(see e.g. \cite{Miller},\cite{Fedoryuk}).
Also, we describe a modification of Pr\"ufer transformation,
introduced in \cite{Glaser} and used in the rest of the paper.

Suppose that in the second order ODE
\begin{align}
\frac{ d }{ dt } \brk{ p(t) u'(t) } + q(t) u(t) = 0
\label{eq_prufer_0}
\end{align}
the variable $t$ varies over some interval $I$ in which $p$ and $q$
are continuously differentiable and have no roots. 
We define the function 
$\theta: I \to \Rc$ via 
\begin{align}
\frac{ p(t) u'(t) }{ u(t) } =
\gamma(t) \tan \theta(t),
\label{eq_general_prufer}
\end{align}
where $\gamma: I \to \Rc$ is an arbitrary 
positive continuously differentiable function.
The function $\theta(t)$ satisfies, for all $t$ in $I$,
\begin{align}
\theta'(t) = -\frac{\gamma(t)}{p(t)} \sin^2 \theta(t)
             -\frac{q(t)}{\gamma(t)} \cos^2 \theta(t)
             -\brk{ \frac{\gamma'(t)}{\gamma(t)}} 
              \frac{\sin\brk{2\theta(t)}}{2}.
\label{eq_general_dtheta}
\end{align}
One can observe that if $u'(\tilde{t}) = 0$ for $\tilde{t} \in I$, then 
by \eqref{eq_general_prufer} 
\begin{align}
\theta(\tilde{t}) = k \pi, \quad
k \text{ is integer}.
\label{eq_prufer_du_zero}
\end{align}
Similarly, if $u(\tilde{t}) = 0$ for $\tilde{t} \in I$, then
\begin{align}
\theta(\tilde{t}) = \brk{k + 1/2} \pi, \quad
k \text{ is integer}.
\label{eq_prufer_u_zero}
\end{align}
The choice $\gamma(t) = 1$ in \eqref{eq_general_prufer}
gives rise to the classical Pr\"ufer transformation (see e.g. 
section 4.2 in \cite{Miller}).

In \cite{Glaser}, the choice $\gamma(t) = \sqrt{q(t) p(t)}$ is
suggested and shown to be more convenient numerically in 
several applications. In this paper, this choice also leads to a
more convenient analytical tool than the classical Pr\"ufer transformation.

Writing \eqref{eq_prolate_ode} in the form of \eqref{eq_prufer_0} yields
\begin{align}
p(t) = 1 - t^2, \quad q(t) = \chi_n - c^2 t^2,
\label{eq_prufer_1}
\end{align}
for $\abrk{t} < \min\cbrk{\sqrt{\chi_n}/c, 1}$.
The equation \eqref{eq_general_prufer} admits the form
\begin{align}
\frac{ p(t) \psi_n'(t) }{ \psi_n(t) } =
\sqrt{ p(t) q(t) } \tan \theta(t),
\label{eq_modified_prufer}
\end{align}
which implies that
\begin{align}
\theta(t) = \text{atan} \brk{ \sqrt{\frac{p(t)}{q(t)}} 
                        \frac{\psi_n'(t)}{\psi_n(t)} } + \pi m(t),
\label{eq_modified_prufer_theta}
\end{align}
where $m(t)$ is an integer determined for all $t$ by an arbitrary choice
at some $t = t_0$ (the role of $\pi m(t)$ in
\eqref{eq_modified_prufer_theta} is to enforce the continuity
of $\theta$ at the roots of $\psi_n$).
The first order ODE \eqref{eq_general_dtheta} admits the form
(see \cite{Glaser}, \cite{Fedoryuk})
\begin{align}
\theta'(t) = -f(t) + \sin \brk{2 \theta(t)} v(t),
\label{eq_jan_prufer_fv}
\end{align}
where the functions $f, v$ are defined, respectively, via the formulae
\begin{align}
f(t) = \sqrt{ \frac{ q(t) }{ p(t) } }
     = \sqrt{ \frac{ \chi_n - c^2 t^2 }{ 1 - t^2 } }
\label{eq_jan_f}
\end{align}
and
\begin{align}
v(t) = -
\frac{1}{4} \cdot \frac{ p(t)q'(t) + q(t)p'(t) }{ p(t) q(t) } 
= 
\frac{1}{2} \brk{ \frac{t}{1-t^2} + \frac{c^2 t}{\chi_n - c^2 t^2} }.
\label{eq_jan_v}
\end{align}

\section{Summary}
\label{sec_summary}
In this section, we summarize some of the properties of
prolate spheroidal wave functions (PSWFs),
proved in Section~\ref{sec_analytical}.
The PSWFs and the related notation were introduced in 
Section~\ref{sec_pswf}.
Throughout this section, the band limit $c > 0$ is
assumed to be a positive real number.

Many properties of the PSWF $\psi_n$ depend on
whether the eigenvalue $\chi_n$ of the ODE \eqref{eq_prolate_ode}
is greater than or less than $c^2$. The following simple relation 
between $c, n$ and $\chi_n$ is
proven in Theorem~\ref{thm_n_and_khi} in Section~\ref{sec_sharp}.
\begin{prop}
Suppose that $n \geq 2$ is a non-negative integer.
\begin{itemize}
\item If $n \leq (2c/\pi)-1$, then $\chi_n < c^2$.
\item If $n \geq (2c/\pi)$, then $\chi_n > c^2$.
\item If $(2c/\pi)-1 < n < (2c/\pi)$, then either inequality is possible.
\end{itemize}
\label{prop_n_and_khi}
\end{prop}

In the following proposition, we describe the location of ``special points''
(roots of $\psi_n$, roots of $\psi_n'$, turning points of the
ODE \eqref{eq_prolate_ode}) that depends on whether $\chi_n > c^2$
or $\chi_n < c^2$. 
It is proven in Lemma~\ref{lem_five} in Section~\ref{sec_special_points}
and is illustrated in
Figures~\ref{fig:test75a}, \ref{fig:test75b}.
\begin{prop}
Suppose that $n \geq 2$ is a positive integer. Suppose
also that $t_1<\dots<t_n$ are the roots of $\psi_n$ in $(-1,1)$,
and $x_1<\dots<x_{n-1}$ are the roots of $\psi_n'$ in $(t_1, t_n)$.
Suppose furthermore that the real number $x_n$ is defined via
the formula
\begin{align}
x_n = 
\begin{cases}
\text{maximal root of } \psi_n' \text{ in } (-1,1),
 & \text{ if } \chi_n < c^2, \\
1, & \text{ if } \chi_n > c^2.
\end{cases}
\label{eq_xn_prop}
\end{align}
Then,
\begin{align}
-\frac{\sqrt{\chi_n}}{c} < -x_n < 
t_1 < x_1 < t_2 < \dots < t_{n-1} < t_n < x_n < \frac{\sqrt{\chi_n}}{c}.
\label{eq_all_tx_prop}
\end{align}
In particular, if $\chi_n < c^2$, then
\begin{align}
t_n < x_n < \frac{\sqrt{\chi_n}}{c} < 1,
\label{eq_khi_small_prop}
\end{align}
and $\psi_n'$ has $n+1$ roots in the interval $(-1,1)$; and,
if $\chi_n > c^2$, then
\begin{align}
t_n < x_n = 1 < \frac{\sqrt{\chi_n}}{c},
\label{eq_khi_large_prop}
\end{align}
and $\psi_n'$ has $n-1$ roots in the interval $(-1,1)$.
\label{prop_five}
\end{prop}
The following two inequalities improve the inequality
\eqref{eq_khi_crude} in Section~\ref{sec_pswf}. Their proof can be found
in Theorems~\ref{thm_n_upper},\ref{thm_n_lower} in Section~\ref{sec_sharp}.
This is one of the principal
analytical results of this paper.
The inequalities \eqref{eq_jan_n_both_large_prop}, 
\eqref{eq_jan_n_both_small_prop} below
are illustrated in Tables~\ref{t:test77a}, \ref{t:test77b}, \ref{t:test77c},
\ref{t:test99}.
\begin{prop}
Suppose that $n \geq 2$ is a positive integer. 
Suppose also that
$t_n$ and $T$ are 
the maximal roots of $\psi_n$ and $\psi_n'$ in the interval $\brk{-1, 1}$, 
respectively.
If $\chi_n > c^2$, then
\begin{align}
1 + \frac{2}{\pi} \int_0^{t_n} \sqrt{ \frac{\chi_n - c^2 t^2}{1 - t^2} } \; dt
< n <
\frac{2}{\pi} \int_0^1 \sqrt{ \frac{\chi_n - c^2 t^2}{1 - t^2} } \; dt.
\label{eq_jan_n_both_large_prop}
\end{align}
If $\chi_n < c^2$, then
\begin{align}
1 + \frac{2}{\pi} \int_0^{t_n} \sqrt{ \frac{\chi_n - c^2 t^2}{1 - t^2} } \; dt
< n <
\frac{2}{\pi} \int_0^{T} 
 \sqrt{ \frac{\chi_n - c^2 t^2}{1 - t^2} } \; dt.
\label{eq_jan_n_both_small_prop}
\end{align}
Note that \eqref{eq_jan_n_both_large_prop} 
and \eqref{eq_jan_n_both_small_prop}
differ only in
the range of integration on their right-hand sides.
\label{prop_jan_n_both}
\end{prop}
In the following proposition, we simplify the inequality
\eqref{eq_jan_n_both_large_prop} in Proposition~\ref{prop_jan_n_both}.
It is proven in Theorem~\ref{thm_n_khi_simple} and 
Corollary~\ref{cor_khi_simple}
in Section~\ref{sec_first_order}.
\begin{prop}
Suppose that $n \geq 2$ is a positive integer, and that $\chi_n > c^2$. Then,
\begin{align}
n < & \; 
\frac{2}{\pi} \int_0^1 \sqrt{ \frac{\chi_n - c^2 t^2}{1 - t^2} } \; dt
= \nonumber \\
& \; \frac{2}{\pi} \sqrt{\chi_n} \cdot E \brk{ \frac{c}{\sqrt{\chi_n}} }
< n+3,
\label{eq_both_large_simple_prop}
\end{align}
where the function $E:\left[0,1\right] \to \Rc$ is defined via
\eqref{eq_E} in Section~\ref{sec_elliptic}.
\label{prop_n_khi_simple}
\end{prop}
The following proposition is an immediate consequence
of Proposition~\ref{prop_n_khi_simple}.
It is proven in Theorem~\ref{thm_exp_term} in Section~\ref{sec_first_order},
and is illustrated in Figures~\ref{fig:test171a},~\ref{fig:test171b}.
\begin{prop}
Suppose that $n \geq 2$ is a positive integer such that
$n > 2c/\pi$, and that the function
$f:[0,\infty) \to \Rc$ is defined via the formula
\begin{align}
f(x) = -1 + \int_0^{\pi/2} \sqrt{ x + \cos^2(\theta)} \; d\theta.
\label{eq_f_def_prop}
\end{align}
Suppose also that the function $H: [0,\infty) \to \Rc$ is the inverse of $f$,
in other words, 
\begin{align}
y = f(H(y)) = 
-1 + \int_0^{\pi/2} \sqrt{ H(y) + \cos^2(\theta)} \; d\theta,
\label{eq_big_h_def_prop}
\end{align}
for all real $y \geq 0$.
Then,
\begin{align}
H\left( \frac{n\pi}{2c} - 1 \right) < 
\frac{\chi_n - c^2}{c^2} < 
H\left( \frac{n\pi}{2c} - 1 + \frac{3\pi}{2c} \right).
\label{eq_khi_via_h_prop}
\end{align}
\label{prop_exp_term}
\end{prop}
In the following proposition, we describe a relation
between $\chi_n$ and the maximal root $t_n$ of $\psi_n$
in $(-1,1)$, by providing lower and upper bounds
on $1-t_n$ in terms of $\chi_n$ and $c$.
It is proven in Theorem~\ref{thm_tn_upper}, \ref{thm_tn_lower}
in Section~\ref{sec_first_order}.
\begin{prop}
Suppose that $n \geq 2$ is a positive integer, and
that $\chi_n > c^2$. Suppose also that $t_n$ is the maximal root
of $\psi_n$ in the interval $(-1,1)$. 
Then,
\begin{align}
\frac{ \pi^2/8 }{\chi_n-c^2 + \sqrt{(\chi_n-c^2)^2+(\pi c/2)^2} } 
& \; <
1-t_n \nonumber \\
& \; < \frac{ 4\pi^2 }{\chi_n-c^2 + \sqrt{(\chi_n-c^2)^2+(4\pi c)^2} }.
\label{eq_tn_both}
\end{align}
\label{prop_tn_both}
\end{prop}
The following proposition is a special case of
Proposition~\ref{prop_tn_both}. It is proven
in Theorem~\ref{thm_tn_simple} in Section~\ref{sec_first_order},
and is illustrated in Figure~\ref{fig:test170}.
\begin{prop}
Suppose that $c > 10/\pi$.
Suppose also that $n \geq 2$ is a positive integer, and that
\begin{align}
n > \frac{2c}{\pi} + 1 + \frac{1}{4} \cdot \log(c).
\label{eq_tn_simple_2_prop}
\end{align}
Suppose furthermore that $t_n$ is the maximal root of $\psi_n$
in the interval $(-1,1)$. Then,
\begin{align}
\frac{\pi^2}{8 \cdot (1+\sqrt{2})} \cdot \frac{1}{\chi_n-c^2} <
1-t_n <
\frac{2\pi^2}{\chi_n - c^2}.
\label{eq_tn_simple_prop}
\end{align}
\label{prop_tn_simple}
\end{prop}
In the following proposition, 
we provide yet another upper bound on $\chi_n$ in terms of $n$.
Its proof can be found in
Theorem~\ref{thm_khi_n_square} in Section~\ref{sec_first_order},
and it is illustrated 
in Tables~\ref{t:test80a},~\ref{t:test80b}
and Figure~\ref{fig:test171a}.
\begin{prop}
Suppose that $n \geq 2$ is a positive integer, and that $\chi_n > c^2$. Then,
\begin{align}
\chi_n < \brk{ \frac{\pi}{2} \brk{n+1} }^2.
\label{eq_khi_n_square_prop}
\end{align}
\label{prop_khi_n_square}
\end{prop}
We observe that, 
for sufficiently large $n$, 
the inequality
\eqref{eq_khi_n_square_prop} is even weaker than
\eqref{eq_khi_crude}.
On the other hand, \eqref{eq_khi_n_square_prop}
can be useful for $n$ near
$2c/\pi$,
as illustrated in
Tables~\ref{t:test80a}, \ref{t:test80b}.

The following proposition summarizes
Theorem~\ref{thm_zeros_inside_bounds} in Section~\ref{sec_sharp}
and Theorems~\ref{thm_spacing_khi},
\ref{thm_spacing_inside}
in Section~\ref{sec_first_order}.
It is illustrated in
Tables~\ref{t:test80a}, \ref{t:test80b},
\ref{t:test98a}, \ref{t:test98b}.
\begin{prop}
Suppose that $n \geq 2$ is a positive integer,
and that $\chi_n > c^2$. Suppose also that
$-1 < t_1 < t_2 < \dots < t_n < 1$ are the roots of $\psi_n$
in the interval $(-1,1)$. Suppose furthermore
that the functions $f,v$ are defined,
respectively, via \eqref{eq_jan_f},\eqref{eq_jan_v}
in Section~\ref{sec_prufer}.
Then:
\begin{itemize}
\item For each integer $(n+1)/2 \leq i \leq n-1$, i.e. for each integer $i$
such that $0 \leq t_i < t_n$,
\begin{align}
\frac{\pi}{ f(t_{i+1}) + v(t_{i+1})/2 } <
t_{i+1} - t_{i} <
\frac{\pi}{ f(t_i) }.
\label{eq_zeros_inside_bounds_prop}
\end{align}
\item For each integer $(n+1)/2 \leq i \leq n-1$, i.e. for each integer $i$
such that $0 \leq t_i < t_n$,
\begin{align}
t_{i+1} - t_i > t_{i+2} - t_{i+1} > \dots > t_n - t_{n-1}.
\label{eq_spacing_shrink_prop}
\end{align}
\item For all integer $j = 1, \dots, n-1$,
\begin{align}
t_{j+1} - t_j < \frac{ \pi }{ \sqrt{\chi_n + 1} }.
\end{align}
\end{itemize}
\label{prop_spacing_big}
\end{prop}
The following proposition summarizes
Theorem~\ref{thm_spacing_inside}
in Section~\ref{sec_first_order}.
\begin{prop}
Suppose that $n \geq 2$ is an integer,
and that $\chi_n < c^2 - c\sqrt{2}$. Suppose also that
$-1 < t_1 < t_2 < \dots < t_n < 1$ are the roots of $\psi_n$
in the interval $(-1,1)$. 
Then,
\begin{align}
t_{i+1} - t_i < t_{i+2} - t_{i+1} < \dots < t_n - t_{n-1},
\label{eq_spacing_stretch_prop}
\end{align}
for all integer $(n+1)/2 \leq i \leq n-1$, i.e. for all integer $i$
such that $0 \leq t_i < t_n$.
\label{prop_spacing_small}
\end{prop}
The following proposition summarizes
Theorem~\ref{thm_psi1_upper_bound} in Section~\ref{sec_pswf}
and Theorem~\ref{thm_psi1} in Section~\ref{sec_growth}.
\begin{prop}
Suppose that $n \geq 0$ is a non-negative integer, and that
$\chi_n > c^2$. Then,
\begin{align}
\frac{1}{2} < \psi_n^2(1) < n + \frac{1}{2}.
\end{align}
\label{prop_psi1}
\end{prop}
The following proposition is
illustrated in Figures~\ref{fig:test75a}, \ref{fig:test75b}.
It is proven in Theorem~\ref{thm_extrema} in Section~\ref{sec_first_order}.
\begin{prop}
Suppose that $n \geq 0$ is a non-negative integer, and that
$x, y$ are two arbitrary extremum points of $\psi_n$ in $(-1,1)$.
If $\abrk{x} < \abrk{y}$, then
\begin{align}
\abrk{\psi_n(x)} < \abrk{\psi_n(y)}.
\label{eq_extremum_general_prop}
\end{align}
If, in addition, $\chi_n > c^2$, then
\begin{align}
\abrk{\psi_n(x)} < \abrk{\psi_n(y)} < \abrk{\psi_n(1)}.
\label{eq_extremum_special_prop}
\end{align}
\label{prop_extrema}
\end{prop}

\section{Analytical Apparatus}
\label{sec_analytical}
The purpose of this section is to provide
the analytical apparatus to be used in the
rest of the paper, as well as to prove the results
summarized in Section~\ref{sec_summary}.

\subsection{Oscillation Properties of PSWFs}
\label{sec_oscillation}
In this subsection, we prove several facts
about the distance between consecutive roots of PSWFs 
\eqref{eq_prolate_integral} and find a more
subtle relationship between $n$ and $\chi_n$ \eqref{eq_prolate_ode}
than the one given by \eqref{eq_khi_crude}. Throughout this subsection
$c > 0$ is a positive real number 
and $n$ is a non-negative integer.
The principal results of this subsection are
Theorems~\ref{thm_n_upper}, \ref{thm_n_lower},
\ref{thm_good_n_good_khi},
and
\ref{thm_khi_n_square}.

\subsubsection{Special Points of $\psi_n$}
\label{sec_special_points}
We refer to the roots of $\psi_n$, the roots of $\psi_n'$ and
the turning points of the ODE \eqref{eq_prolate_ode} as 
''special points''. Some of them 
play an important role in the subsequent analysis.
These points are introduced in the following definition.
\begin{definition}[Special points]
Suppose that $n \geq 2$ is a positive integer. We define
\begin{itemize}
\item $t_1 < t_2 < \dots < t_n$ to be the roots of $\psi_n$ in $(-1,1)$,
\item $x_1 < \dots <x_{n-1}$ to be the roots of $\psi_n'$ in $(t_1,t_n)$,
\item $x_n$ via the formula
\begin{align}
x_n = 
\begin{cases}
\text{maximal root of } \psi_n' \text{ in } (-1,1),
 & \text{ if } \chi_n < c^2, \\
1, & \text{ if } \chi_n > c^2.
\end{cases}
\label{eq_xn}
\end{align}
\end{itemize}
\label{def_special}
\end{definition}
\noindent
This definition will be used throughout all of
Section~\ref{sec_analytical}. The relative location of some of 
the special points
depends on whether $\chi_n > c^2$ or $\chi_n < c^2$.
This is illustrated in
Figures~\ref{fig:test75a}, \ref{fig:test75b}
and is described by the following lemma.
\begin{lem}[Special points]
Suppose that $n \geq 2$ is a positive integer. Suppose
also that $t_1<\dots<t_n$ and $x_1<\dots<x_n$
are those of Definition~\ref{def_special}.
Then,
\begin{align}
-\frac{\sqrt{\chi_n}}{c} < -x_n < 
t_1 < x_1 < t_2 < \dots < t_{n-1} < t_n < x_n < \frac{\sqrt{\chi_n}}{c}.
\label{eq_all_tx}
\end{align}
In particular, if $\chi_n < c^2$, then
\begin{align}
t_n < x_n < \frac{\sqrt{\chi_n}}{c} < 1,
\label{eq_khi_small}
\end{align}
and $\psi_n'$ has $n+1$ roots in the interval $(-1,1)$; and,
if $\chi_n > c^2$, then
\begin{align}
t_n < x_n = 1 < \frac{\sqrt{\chi_n}}{c},
\label{eq_khi_large}
\end{align}
and $\psi_n'$ has $n-1$ roots in the interval $(-1,1)$.
\label{lem_five}
\end{lem}
\begin{proof}
Without loss of generality, we assume that
\begin{align}
\psi_n(1) > 0.
\label{eq_psi_at_1}
\end{align}
Obviously, \eqref{eq_psi_at_1} implies that
\begin{align}
\psi_n'(t_n) > 0.
\label{eq_dpsi_at_tn}
\end{align}
Suppose first that $\chi_n < c^2$. Then, due to the ODE \eqref{eq_prolate_ode}
in Section~\ref{sec_pswf},
\begin{align}
\psi_n'(1) = \frac{\chi_n-c^2}{2} \cdot \psi_n(1) < 0.
\label{eq_dpsi_at_1}
\end{align}
We combine \eqref{eq_prolate_ode} and \eqref{eq_dpsi_at_tn}
to obtain
\begin{align}
\psi_n''(t_n) = \frac{2t_n}{1-t_n^2} \cdot \psi_n'(t_n) > 0.
\label{eq_ddpsi_at_tn}
\end{align}
In addition, we combine \eqref{eq_xn}, \eqref{eq_dpsi_at_tn},
\eqref{eq_dpsi_at_1}
to conclude that the maximal root $x_n$
of $\psi_n'$ in $(-1,1)$ satisfies 
\begin{align}
t_n < x_n < 1.
\label{eq_tn_xn_1}
\end{align}
Moreover, \eqref{eq_psi_at_1} implies that,
for any root $x$ of $\psi_n'$ in $(t_n,1)$,
\begin{align}
\psi_n''(x) = -\frac{\chi_n-c^2 x^2}{1-x^2} \cdot \psi_n(x) < 0.
\label{eq_ddpsi_at_xn}
\end{align}
We combine \eqref{eq_psi_at_1}, \eqref{eq_tn_xn_1}, \eqref{eq_ddpsi_at_xn}
with \eqref{eq_prolate_ode} to obtain
\begin{align}
\frac{c^2 x_n^2 - \chi_n}{1 - x_n^2} = \frac{\psi_n''(x_n)}{\psi_n(x_n)} < 0,
\end{align}
which implies both \eqref{eq_all_tx} and \eqref{eq_khi_small}. In addition, 
we combine \eqref{eq_prolate_ode}, \eqref{eq_khi_small} 
and
\eqref{eq_ddpsi_at_xn}
to conclude
that $x_n$ is the only root of $\psi_n'$ between $t_n$ and 1.
Thus, $\psi_n'$ indeed has $n+1$ roots in $(-1,1)$.

Suppose now that $\chi_n > c^2$. We combine \eqref{eq_prolate_ode}
and \eqref{eq_psi_at_1} to obtain
\begin{align}
\psi_n'(1) = \frac{\chi_n-c^2}{2} \cdot \psi_n(1) > 0.
\label{eq_dpsi_at_1a}
\end{align}
If $t_n < x < 1$ is a root of $\psi_n'$, then
\begin{align}
\psi_n''(x) = -\frac{\chi_n-c^2x^2}{1-x^2} \cdot \psi_n(x) < 0,
\label{eq_ddpsi_at_x}
\end{align}
therefore $\psi_n'$ can have at most one root in $(t_n,1)$. 
We combine this observation with
\eqref{eq_xn}, \eqref{eq_dpsi_at_tn}, \eqref{eq_dpsi_at_1a}  and
\eqref{eq_ddpsi_at_x} to conclude that, in fact, $\psi_n'$ has
no roots in $(t_n,1)$, and hence both
\eqref{eq_all_tx} and \eqref{eq_khi_large} hold. In particular, $\psi_n'$
has $n-1$ roots in $(-1,1)$.
\end{proof}
\subsubsection{A Sharper Inequality for $\chi_n$}
\label{sec_sharp}
In this subsection,
we use the modified Pr\"ufer transformation
(see Section~\ref{sec_prufer}) to analyze
the relationship between $n, c$ and $\chi_n$. In particular,
this analysis yields
fairly tight lower and upper bounds on $\chi_n$
in terms of $c$ and $n$. These bounds are described in
Theorems~\ref{thm_n_upper},\ref{thm_n_lower} 
below. These theorems are not only one of the principal
results of this paper, but are also subsequently used in the proofs of
Theorems~\ref{thm_zeros_inside_bounds},
\ref{thm_good_n_good_khi}, \ref{thm_khi_n_square},
\ref{thm_n_khi_simple}, \ref{thm_tn_lower}.

We start with developing the required analytical machinery.
In the following lemma,
we describe several properties of the modified
Pr\"ufer transformation (see Section~\ref{sec_prufer}),
applied to the prolate differential equation \eqref{eq_prolate_ode}.
\begin{lem}
Suppose that $n \geq 2$ is a positive integer. Suppose
also that the numbers $t_1,\dots,t_n$ and $x_1,\dots,x_n$
are those of Definition~\ref{def_special} in Section~\ref{sec_special_points},
and that the function $\theta: \left[-x_n,x_n\right] \to \Rc$ is
defined via the formula
\begin{align}
\theta(t) = 
\begin{cases}
\left(i-\frac{1}{2}\right) \cdot \pi, & \text{ if } t = t_i
\text{ for some } 1 \leq i \leq n, \\
\\
\text{atan}\left( -\sqrt{\frac{1-t^2}{\chi_n-c^2t^2}} \cdot
                        \frac{\psi_n'(t)}{\psi_n(t)} \right)
          + m(t) \cdot \pi, & \text{ if } \psi_n(t) \neq 0, \\
\end{cases}
\label{eq_prufer_theta}
\end{align}
where $m(t)$ is the number of the roots
of $\psi_n$ in the interval $(-1,t)$.
Then, $\theta$ has
the following properties:
\begin{itemize}
\item $\theta$ is continuously differentiable in 
the interval $\left[-x_n,x_n\right]$.
\item $\theta$ satisfies, for all $-x_n < t < x_n$, the
differential equation
\begin{align}
\theta'(t) = f(t) + v(t) \cdot \sin( 2 \theta(t) ),
\label{eq_prufer_theta_ode}
\end{align}
where the functions $f,v$ are defined, respectively, 
via \eqref{eq_jan_f}, \eqref{eq_jan_v} in
Section~\ref{sec_prufer}.
\item for each integer $0 \leq j \leq 2n$, there is a unique solution
to the equation
\begin{align}
\theta(t) = j \cdot \frac{\pi}{2},
\label{eq_prufer_sm2_equation}
\end{align}
for the unknown $t$ in $\left[-x_n,x_n\right]$.
More specifically,
\begin{align}
\label{eq_theta_at_xn}
& \theta(-x_n) = 0, \\
\label{eq_theta_at_t}
& \theta(t_i) = \left(i-\frac{1}{2}\right) \cdot \pi, \\
\label{eq_theta_at_x}
& \theta(x_i) = i \cdot \pi,
\end{align}
for each $i = 1, \dots, n$. In particular, $\theta(x_n) = n \cdot \pi$.
\end{itemize}
\label{lem_prufer}
\end{lem}
\begin{proof}
We combine \eqref{eq_all_tx} in 
Lemma~\ref{lem_five} with \eqref{eq_prufer_theta} to conclude that
$\theta$
is well defined for all $-x_n \leq t \leq x_n$,
where $x_n$ is
given via \eqref{eq_xn} in Definition~\ref{def_special}.
Obviously, $\theta$ is continuous, and the identities
\eqref{eq_theta_at_xn}, \eqref{eq_theta_at_t}, \eqref{eq_theta_at_x}
follow immediately from the combination
of Lemma~\ref{lem_five} and \eqref{eq_prufer_theta}.
In addition, $\theta$ satisfies the ODE
\eqref{eq_prufer_theta_ode} in $(-x_n,x_n)$ due to
\eqref{eq_general_dtheta}, \eqref{eq_modified_prufer},
\eqref{eq_jan_prufer_fv} in Section~\ref{sec_prufer}.

Finally, to establish the uniqueness of
the solution to the equation \eqref{eq_prufer_sm2_equation},
we make the following observation. Due to \eqref{eq_prufer_theta},
for any point $t$ in $(-x_n,x_n)$, the value $\theta(t)$ is an integer
multiple of $\pi/2$ if and only if $t$ is either a root of $\psi_n$
or a root of $\psi_n'$. 
We conclude the proof by combining this observation with 
\eqref{eq_xn},
\eqref{eq_theta_at_xn} and \eqref{eq_theta_at_x}.
\end{proof}
\begin{remark}
We observe that, due to \eqref{eq_theta_at_xn}, \eqref{eq_theta_at_t},
\eqref{eq_theta_at_x},
for all $i = 1, \dots, n$,
\begin{align}
\sin(2\theta(t_i)) = \sin(2\theta(x_i)) = 0,
\end{align}
where $t_1,\dots,t_n$, $x_1,\dots,x_n$ are those of
Definition~\ref{def_special} in Section~\ref{sec_special_points},
and $\theta$ is that of Lemma~\ref{lem_prufer}. This observation
will play an important role in the analysis of the ODE
\eqref{eq_prufer_theta_ode} throughout the rest of this section.
\end{remark}

In the following lemma, we prove that $\theta$ of Lemma~\ref{lem_prufer}
is monotonically increasing.
\begin{lem}
Suppose that $n \geq 2$ is a positive integer.
Suppose also that the real number $x_n$
and the function $\theta: \left[-x_n,x_n\right] \to \Rc$
are those of Lemma~\ref{lem_prufer} above. Then, 
$\theta$ is strictly increasing in $\left[-x_n,x_n\right]$,
in other words, 
\begin{align}
\theta'(t) > 0,
\label{eq_theta_increasing}
\end{align}
for all $-x_n < t < x_n$.
\label{lem_theta_increasing}
\end{lem}
\begin{proof}
We first prove that
\begin{align}
\frac{d}{dt} \left( \frac{v}{f} \right)(t) > 0,
\label{eq_dvf_positive}
\end{align}
for $-x_n < t < x_n$, where the functions $f,v$ are defined,
respectively, via \eqref{eq_jan_f}, \eqref{eq_jan_v}
in Section~\ref{sec_prufer}. We differentiate $v/f$ with respect
to $t$ to obtain
\begin{align}
\brk{ \frac{ v }{ f } }'
& \; = - \brk{ \frac{ p'q + q'p }{ 4pq } \cdot \sqrt{ \frac{p}{q} } }'
     = - \brk{ \frac{ p'q + q'p }{ 4q^{3/2} p^{1/2} } }' \nonumber \\
& \; = \frac{q^{-3} p^{-1}}{4} \cdot \left[
  \brk{\frac{3}{2} q^{1/2} p^{1/2} q' + \frac{1}{2} q^{3/2} p^{-1/2} p'}
  \brk{p'q + q'p} - \right. \nonumber \\
& \quad \quad \quad \quad \quad \; \left. 
  \brk{p''q + 2p'q' + q''p} q^{3/2} p^{1/2} \right] \nonumber \\
& \; = \frac{q^{-5/2} p^{-3/2}}{4} \cdot \sbrk{
  \brk{\frac{3}{2} q'p + \frac{1}{2} p'q} \brk{p'q + q'p} -
  pq \brk{p''q + 2p'q' + q''p} } \nonumber \\
& \; = \frac{q^{-5/2} p^{-3/2}}{4} \cdot \sbrk{
  \frac{3}{2} p^2 \brk{q'}^2 + \frac{1}{2} q^2 \brk{p'}^2 - 
  q^2 p p'' - p^2 q q'' } > 0,
\label{eq_d4vf}
\end{align}
since, due to \eqref{eq_prufer_1},
\begin{align}
p(t) > 0, \quad p''(t) = -2 < 0, \quad
q(t) > 0, \quad q''(t) = -2 c^2 < 0.
\end{align}
We now proceed to prove \eqref{eq_theta_increasing} for
$0 < t < x_n$.
Suppose that, by contradiction,
there exists $0 < x < x_n$ such that
\begin{align}
\theta'(x) < 0.
\label{eq_theta_x_contra}
\end{align}
Combined with \eqref{eq_prufer_theta_ode} in Lemma~\ref{lem_prufer} above,
\eqref{eq_theta_x_contra} implies that
\begin{align}
1 + \frac{v(x)}{f(x)} \cdot \sin(2\theta(x)) =
\frac{f(x) + v(x) \cdot \sin(2\theta(x))}{f(x)} < 0,
\label{eq_dtheta_contra}
\end{align}
and, in particular, that
\begin{align}
\sin(2\theta(x)) < 0.
\label{eq_sin_contra}
\end{align}
Due to Lemma~\ref{lem_prufer} above,
there exists an integer
$(n+1)/2 \leq i \leq n$ such that
\begin{align}
\left(i-\frac{1}{2}\right) \cdot \pi  < \theta(x) < i \cdot \pi.
\label{eq_thetax_squeezed}
\end{align}
Moreover, due to \eqref{eq_theta_at_xn}, 
\eqref{eq_theta_at_t}, \eqref{eq_theta_at_x}, 
\eqref{eq_theta_x_contra}, \eqref{eq_sin_contra},
\eqref{eq_thetax_squeezed}, there exists a point $y$ such that
\begin{align}
0 \leq t_i < x < y < x_i \leq x_n,
\end{align}
and also
\begin{align}
\theta(x) = \theta(y), \quad \theta'(y) > 0.
\label{eq_theta_xy}
\end{align}
for otherwise \eqref{eq_theta_at_x} would be impossible.
We combine \eqref{eq_prufer_theta_ode} and \eqref{eq_theta_xy} to obtain
\begin{align}
1 + \frac{v(y)}{f(y)} \cdot \sin(2\theta(x)) =
\frac{f(y) + v(y) \cdot \sin(2\theta(y)) }{f(y)} =
\frac{\theta'(y)}{f(y)} > 0,
\end{align}
in contradiction to \eqref{eq_dvf_positive}, \eqref{eq_dtheta_contra} and
\eqref{eq_sin_contra}. This concludes the proof
of \eqref{eq_theta_increasing} for $0 < t < x_n$.
For $-x_n < t < 0$, the identity
\eqref{eq_theta_increasing} follows now from the symmetry considerations.
\end{proof}
The right-hand side of the ODE \eqref{eq_prufer_theta_ode}
of Lemma~\ref{lem_prufer} contains a monotone term and
an oscillatory term. In the following lemma,
we study the integrals of the oscillatory term between various special
points, introduced in Definition~\ref{def_special} 
in Section~\ref{sec_special_points}.
\begin{lem}
Suppose that $n \geq 2$ is an integer.
Suppose also that the real numbers $t_1<\dots<t_n$ and $x_1<\dots <x_n$
are those of 
Definition~\ref{def_special} in Section~\ref{sec_special_points},
and the function $\theta: \left[-x_n, x_n\right] \to \Rc$ is
that of Lemma~\ref{lem_prufer} above. Suppose furthermore
that the function $v$ is defined via \eqref{eq_jan_v}
in Section~\ref{sec_prufer}. Then,
\begin{align}
\label{eq_int_x_tp}
& \int_{x_i}^{t_{i+1}} v(t) \cdot \sin(2\theta(t)) \; dt > 0, \\
\label{eq_int_tp_xp}
& \int_{t_{i+1}}^{x_{i+1}} v(t) \cdot \sin(2\theta(t)) \; dt < 0, \\
\label{eq_int_x_xp}
& \int_{x_i}^{x_{i+1}} v(t) \cdot \sin(2\theta(t)) \; dt < 0,
\end{align}
for all integer $(n-1)/2 \leq i \leq n-1$, i.e. for all integer $i$
such that $0 \leq x_i < x_n$. Note that the integral in \eqref{eq_int_x_xp}
is the sum of the integrals in \eqref{eq_int_x_tp} and \eqref{eq_int_tp_xp}.
\label{lem_vsin_1}
\end{lem}
\begin{proof}
Suppose that $i$ is a positive integer such that $(n-1)/2 \leq i \leq n-1$.
Suppose also that the function
$s: \left[0, n\cdot\pi\right] \to \left[-x_n,x_n\right]$ 
is the 
inverse of $\theta$. In other words, for all $0 \leq \eta \leq n \cdot \pi$,
\begin{align}
\theta(s(\eta)) = \eta.
\label{eq_s_eta}
\end{align}
Using \eqref{eq_prufer_theta_ode}, \eqref{eq_theta_at_t},
\eqref{eq_theta_at_x} in Lemma~\ref{lem_prufer}, we expand
the left-hand side of \eqref{eq_int_x_tp} to obtain
\begin{align}
& \int_{x_i}^{t_{i+1}} v(t) \cdot \sin(2\theta(t)) \; dt = \nonumber \\
& \int_{\theta(x_i)}^{\theta(t_{i+1})}
   v(s(\eta)) \cdot \sin(2\eta) \cdot s'(\eta) \; d\eta = \nonumber \\
& \int_{i \cdot \pi}^{(i+1/2) \cdot \pi}
   \frac{v(s(\eta)) \cdot \sin(2\eta) \; d\eta}
   {f(s(\eta)) + v(s(\eta)) \cdot \sin(2\eta)} = \nonumber \\
& \int_0^{\pi/2}
   \frac{v(s(\eta+i\cdot\pi)) \cdot \sin(2\eta) \; d\eta}
   {f(s(\eta+i\cdot\pi)) + v(s(\eta+i\cdot\pi)) \cdot \sin(2\eta)},
\label{eq_vsin_a}
\end{align}
from which \eqref{eq_int_x_tp} readily follows due to
\eqref{eq_jan_v} in Section~\ref{sec_prufer} and 
\eqref{eq_theta_increasing} in Lemma~\ref{lem_theta_increasing}.
By the same token, we expand the left-hand side of \eqref{eq_int_tp_xp}
to obtain
\begin{align}
& \int_{t_{i+1}}^{x_{i+1}} v(t) \cdot \sin(2\theta(t)) \; dt = \nonumber \\
& \int_{(i +1/2) \cdot \pi}^{(i+1) \cdot \pi}
   \frac{v(s(\eta)) \cdot \sin(2\eta) \; d\eta}
   {f(s(\eta)) + v(s(\eta)) \cdot \sin(2\eta)} = \nonumber \\
& -\int_0^{\pi/2}
   \frac{v(s(\eta+(i+1/2)\cdot\pi)) \cdot \sin(2\eta) \; d\eta}
   {f(s(\eta+(i+1/2)\cdot\pi)) - v(s(\eta+(i+1/2)\cdot\pi)) \cdot \sin(2\eta)},
\label{eq_vsin_b}
\end{align}
which, combined with
\eqref{eq_jan_v} in Section~\ref{sec_prufer} and 
\eqref{eq_theta_increasing} in Lemma~\ref{lem_theta_increasing},
implies \eqref{eq_int_tp_xp}. Finally, 
for all $0 < \eta < \pi/2$,
\begin{align}
\frac{\sin(2\eta)}{(f/v)( s(\eta+(i+1/2)\cdot\pi) ) - \sin(2\eta)} >
\frac{\sin(2\eta)}{(f/v)( s(\eta + i\cdot\pi) ) + \sin(2\eta)},
\label{eq_vsin_c}
\end{align}
since the function $f/v$ is decreasing due to
\eqref{eq_dvf_positive}
in the proof of Lemma~\ref{lem_theta_increasing}.
The inequality \eqref{eq_int_x_xp} now follows from the combination
of \eqref{eq_vsin_a}, \eqref{eq_vsin_b} and \eqref{eq_vsin_c}.
\end{proof}
We are now ready to prove one of the principal results of this paper.
It is illustrated in
Tables~\ref{t:test77a}, \ref{t:test77b}, \ref{t:test77c}, \ref{t:test99}.
\begin{thm}
Suppose that $n \geq 2$ is a positive integer. 
If $\chi_n > c^2$, then
\begin{align}
n <
\frac{2}{\pi} \int_0^1 \sqrt{ \frac{\chi_n - c^2 t^2}{1 - t^2} } \; dt.
\label{eq_n_upper_large}
\end{align}
If $\chi_n < c^2$, then
\begin{align}
n <
\frac{2}{\pi} \int_0^{T} 
   \sqrt{ \frac{\chi_n - c^2 t^2}{1 - t^2} } \; dt,
\label{eq_n_upper_small}
\end{align}
where $T$
is the maximal root of $\psi_n'$ in $(-1,1)$.
Note that \eqref{eq_n_upper_large} and
\eqref{eq_n_upper_small} differ only in the range of integration
on their right-hand sides.
\label{thm_n_upper}
\end{thm}
\begin{proof}
Suppose that the real numbers
\begin{align}
-1 \leq -x_n < t_1 < x_1 < t_2 < \dots < t_{n-1} < x_{n-1} < t_n < x_n \leq 1
\label{eq_all_xt}
\end{align}
are those of
Definition~\ref{def_special} in Section~\ref{sec_special_points},
and the function $\theta: \left[-x_n,x_n\right] \to \Rc$ is
that of Lemma~\ref{lem_prufer} above. Suppose also that
the functions $f,v$ are defined, respectively, via 
\eqref{eq_jan_f}, \eqref{eq_jan_v} in Section~\ref{sec_prufer}.
If $n$ is even, then we combine
\eqref{eq_prufer_theta_ode}, \eqref{eq_theta_at_t},
\eqref{eq_theta_at_x} in Lemma~\ref{lem_prufer} with
\eqref{eq_int_x_xp} in Lemma~\ref{lem_vsin_1} to obtain
\begin{align}
\frac{n}{2} \cdot \pi & \; = 
\int_{x_{n/2}}^{x_n} \theta'(t) \; dt =
\int_0^{x_n} f(t) \; dt + 
\sum_{i=n/2}^{n-1} \int_{x_i}^{x_{i+1}} v(t) \cdot \sin(2\theta(t)) \; dt
\nonumber \\
& \; < \int_0^{x_n} f(t) \; dt.
\label{eq_n_upper_even}
\end{align}
If $n$ is odd, then we combine
\eqref{eq_prufer_theta_ode}, \eqref{eq_theta_at_t},
\eqref{eq_theta_at_x} in Lemma~\ref{lem_prufer} with
\eqref{eq_int_tp_xp}, \eqref{eq_int_x_xp} in Lemma~\ref{lem_vsin_1} to obtain
\begin{align}
\frac{n}{2} \cdot \pi & \; = 
\int_{t_{(n+1)/2}}^{x_n} \theta'(t) \; dt
= 
\int_0^{x_n} f(t) \; dt \; + \nonumber \\
& \; \; \; \; \; 
\int_{t_{(n+1)/2}}^{x_{(n+1)/2}} v(t) \cdot \sin(2\theta(t)) \; dt +
\sum_{i=(n+1)/2}^{n-1} \int_{x_i}^{x_{i+1}} v(t) \cdot \sin(2\theta(t)) \; dt
\nonumber \\
& \; < \int_0^{x_n} f(t) \; dt.
\label{eq_n_upper_odd}
\end{align}
We combine \eqref{eq_n_upper_even} and \eqref{eq_n_upper_odd}
with \eqref{eq_xn} in Lemma~\ref{lem_five}
to conclude both \eqref{eq_n_upper_large} and
\eqref{eq_n_upper_small}.
\end{proof}
To prove Theorem~\ref{thm_n_lower}, we need to develop a number
of technical tools. In the following two lemmas,
we describe several properties of the equation $f(t)=v(t)$ in the unknown
$t$, where $f,v$ are defined, respectively, 
via \eqref{eq_jan_f}, \eqref{eq_jan_v} in Section~\ref{sec_prufer}.
\begin{lem}
Suppose that $n \geq 0$ is a non-negative integer. Suppose also
that the functions $f, v$ are defined, respectively, 
via \eqref{eq_jan_f}, \eqref{eq_jan_v}
in Section~\ref{sec_prufer}. Suppose furthermore that the
real number $x_n$ is that
of Definition~\ref{def_special} in Section~\ref{sec_special_points}.
Then, there exists a unique point $\hat{t}$ in the interval
$(0,x_n)$ such that
\begin{align}
f(\hat{t}) = v(\hat{t}).
\label{eq_t_hat}
\end{align}
\label{lem_jan_fv}
\end{lem}
\begin{proof}
We observe that, 
due to \eqref{eq_jan_f},\eqref{eq_jan_v} in Section~\ref{sec_prufer},
\begin{align}
\frac{v(t)}{f(t)} > 0
\label{eq_vf_positive}
\end{align}
for all $0 < t < x_n$.
Moreover, 
\begin{align}
\frac{v(0)}{f(0)} = 0, \quad 
\lim_{t \to x_n. \; t < x_n} \frac{v(t)}{f(t)} = \infty.
\label{eq_vf_limit}
\end{align}
We combine
\eqref{eq_dvf_positive} in the proof of Lemma~\ref{lem_theta_increasing}
with \eqref{eq_vf_positive} and \eqref{eq_vf_limit} to conclude
both existence and uniqueness of the solution to the equation $f(t)=v(t)$
in the unknown $0 < t < x_n$.
\end{proof}
\begin{lem}
Suppose that $n \geq 2$ is a positive integer.
Suppose also that the real number $x_n$
and the function $\theta: \left[-x_n,x_n\right] \to \Rc$
are those of Lemma~\ref{lem_prufer} above. Suppose
furthermore that the point $0 < \hat{t} < x_n$ is that
of Lemma~\ref{lem_jan_fv} above.
Then,
\begin{align}
\left(n-\frac{1}{4}\right) \cdot \pi < \theta(\hat{t}) < n \cdot \pi.
\label{eq_pi4_ineq}
\end{align}
\label{lem_pi4}
\end{lem}
\begin{proof}
Suppose that the point $0<x<x_n$ is defined via the formula
\begin{align}
x = \theta^{-1}\left( \left(n-\frac{1}{4}\right) \cdot \pi \right),
\label{eq_x_pi4}
\end{align}
where $\theta^{-1}$ denotes the inverse of $\theta$.
By contradiction, suppose that \eqref{eq_pi4_ineq} does not hold. In other
words,
\begin{align}
0 < \hat{t} < x.
\label{eq_pi4_contra}
\end{align}
It follows from the combination of Lemma~\ref{lem_jan_fv}, 
\eqref{eq_dvf_positive} in the proof of Lemma~\ref{lem_theta_increasing},
and \eqref{eq_pi4_contra}, that $f(x) < v(x)$. On the other hand,
due to \eqref{eq_prufer_theta_ode} in Lemma~\ref{lem_prufer} and 
\eqref{eq_x_pi4},
\begin{align}
\theta'(x)
& \; = f(x) + v(x) \cdot \sin(2\theta(x)) \nonumber \\
& \; = f(x) + v(x) \cdot \sin\left(2 n \pi - \frac{\pi}{2}\right)
     = f(x) - v(x) < 0,
\end{align}
in contradiction to \eqref{eq_theta_increasing} 
in Lemma~\ref{lem_theta_increasing}.
\end{proof}
In the following three lemmas, we study some of the properties
of the ratio $f/v$, where $f,v$ are defined, respectively,
via \eqref{eq_jan_f}, \eqref{eq_jan_v} in Section~\ref{sec_prufer}.
\begin{lem}
Suppose that $n\geq 0$ is a non-negative integer, and that
the functions $f,v$ are defined, respectively, 
via \eqref{eq_jan_f}, \eqref{eq_jan_v}
in Section~\ref{sec_prufer}. Then, for all real $0 < t < 1$,
\begin{align}
-\frac{d}{dt}\left(\frac{f}{v}\right)(t) = h_t(a) \cdot f(t),
\label{eq_dfv}
\end{align}
where the real number $a>0$ is defined via the formula
\begin{align}
a = \frac{\chi_n}{c^2},
\label{eq_dfv_a}
\end{align}
and, for all $0<t<1$, the function $h_t : (0,\infty) \to \Rc$ 
is defined via the formula
\begin{align}
h_t(a) = \frac{ 4t^6 + (2a-6)\cdot t^4 + (4-8a)\cdot t^2 + 
      2a\cdot(a+1) }
     {t^2\cdot\left(1+a-2t^2\right)^2}.
\label{eq_hta}
\end{align}
Moreover, for all real $0 < t < \min\left\{ \sqrt{a}, 1 \right\}$,
\begin{align}
-\frac{d}{dt}\left(\frac{f}{v}\right)(t) \geq \frac{3}{2} \cdot f(t).
\label{eq_dfv_ineq}
\end{align}
\label{lem_dfv}
\end{lem}
\begin{proof}
The identity \eqref{eq_dfv} is obtained from \eqref{eq_jan_f},
\eqref{eq_jan_v}
via straightforward algebraic manipulations. To establish
\eqref{eq_dfv_ineq}, it suffices to show that,
for a fixed $0<t<1$,
\begin{align}
\inf_a \left\{ h_t(a) \; : t^2 < a < \infty \right\} \geq \frac{3}{2}.
\label{eq_hta_1}
\end{align}
We start with observing that, for all $0 < t < 1$,
\begin{align}
\lim_{a \to t^2,\; a > t^2} h_t(a) = 6, \quad
\lim_{a \to \infty} h_t(a) = \frac{2}{t^2}.
\label{eq_hta_2}
\end{align}
Then, we differentiate $h_t(a)$, given via \eqref{eq_hta},
with respect to $a$ to obtain
\begin{align}
\frac{d h_t}{d a}(a) = 
\frac{2 \cdot(1-t^2)}{t^2 \cdot(1+a-2t^2)^3} \cdot
\left(
6t^4 + (a-9)\cdot t^2 + a + 1
\right).
\label{eq_dha}
\end{align}
It follows from \eqref{eq_hta_2}, \eqref{eq_dha}, that if
$t^2 < \hat{a}_t < \infty$ is a local extremum of $h_t(a)$,
then
\begin{align}
\hat{a}_t = \frac{-6t^4+9t^2-1}{t^2+1} > t^2,
\label{eq_ahat}
\end{align}
which is possible if and only if $1 > t^2 > 1/7$. Then we substitute
$\hat{a}_t$, given via \eqref{eq_ahat}, into \eqref{eq_hta} to obtain
\begin{align}
h(t,\hat{a}_t) = \frac{-t^4 + 14t^2-1}{8t^4}.
\label{eq_h_ahat}
\end{align}
It is trivial to verify that
\begin{align}
\inf_t\left\{ h(t,\hat{a}_t) \; : \; \frac{1}{7} < t < 1\right\} =
\lim_{t\to 1, \; t>1} h(t,\hat{a}_t) = \frac{3}{2}.
\label{eq_32}
\end{align}
Now \eqref{eq_hta_1} follows from the combination of
\eqref{eq_hta_2}, \eqref{eq_ahat}, \eqref{eq_h_ahat} and \eqref{eq_32}.
\end{proof}
\begin{lem}
Suppose that $n \geq 2$ is a positive integer, and that $t_n$
is the maximal zero of $\psi_n$ is the interval $(-1,1)$.
Suppose also
that the real number $Z_0$ is defined via the formula
\begin{align}
Z_0 = \frac{1}{1 + \frac{3\pi}{8}} \approx 0.4591.
\label{eq_z0}
\end{align}
Then, for all $0 < t \leq t_n$,
\begin{align}
v(t) < f(t) \cdot Z_0,
\label{eq_fv_z0}
\end{align} 
where the functions $f,v$ are defined, respectively, via
\eqref{eq_jan_f},\eqref{eq_jan_v} in Section~\ref{sec_prufer}. 
\label{lem_z0}
\end{lem}
\begin{proof}
Due to \eqref{eq_dvf_positive} in the proof of Lemma~\ref{lem_theta_increasing},
the function $f/v$ 
decreases monotonically in the interval
$(0,t_n)$, and therefore, to prove \eqref{eq_fv_z0}, it suffices to show
that
\begin{align}
\frac{f(t_n)}{v(t_n)} > \frac{1}{Z_0} = 1 + \frac{3\pi}{8}.
\label{eq_fv_tn}
\end{align}
Suppose that the point $\hat{t}$ is that of Lemma~\ref{lem_jan_fv}.
Suppose also that the real number $x_n$ and the 
function $\theta: \left[-x_n,x_n\right] \to \Rc$ are those
of Lemma~\ref{lem_prufer}. Suppose furthermore that
the function
$s: \left[0, n\cdot\pi\right] \to \left[-x_n,x_n\right]$ 
is the 
inverse of $\theta$. In other words, for all $0 \leq \eta \leq n \cdot \pi$,
\begin{align}
\theta(s(\eta)) = \eta.
\label{eq_s_eta_2}
\end{align}
We combine Lemma~\ref{lem_prufer}, 
Lemma~\ref{lem_theta_increasing},
Lemma~\ref{lem_jan_fv},
Lemma~\ref{lem_pi4} and
Lemma~\ref{lem_dfv} to obtain
\begin{align}
& \frac{f(t_n)}{v(t_n)}-1 =
\left(-\frac{f}{v}\right)(\hat{t}) - \left(-\frac{f}{v}\right)(t_n) =
\nonumber \\
& \int_{t_n}^{\hat{t}} \frac{d}{dt} \left(-\frac{f}{v}\right)(t) \; dt =
 \int_{\theta(t_n)}^{\theta(\hat{t})}
\frac{ \frac{d}{dt}\left(-\frac{f}{v}\right)(s(\eta)) \; d\eta}
{f(s(\eta)) + v(s(\eta)) \cdot \sin(2\eta) } >
\nonumber \\
& \int_{(n-1/2)\pi}^{(n-1/4)\pi}
\frac{ \frac{d}{dt}\left(-\frac{f}{v}\right)(s(\eta)) \; d\eta}
{f(s(\eta)) + v(s(\eta)) \cdot \sin(2\eta) } > 
\nonumber \\
& \int_{(n-1/2)\pi}^{(n-1/4)\pi}
\frac{ \frac{d}{dt}\left(-\frac{f}{v}\right)(s(\eta)) \; d\eta}
{f(s(\eta))} > \frac{\pi}{4} \cdot \frac{3}{2} = \frac{3\pi}{8},
\end{align}
which implies \eqref{eq_fv_tn}.
\end{proof}
\begin{lem}
Suppose that $n \geq 2$ and $(n+1)/2 \leq i \leq n-1$ are
positive integers. Suppose also that the real number $x_n$
and the function $\theta:\left[-x_n,x_n\right] \to \Rc$
are those of Lemma~\ref{lem_prufer}. Suppose furthermore that
$0 < \delta < \pi/4$ is a real number, and that the real
number $Z_\delta$ is defined via the formula
\begin{align}
Z_\delta = \left[
1 + \frac{3}{2} \cdot \left(\frac{\pi}{4} + 
                            \frac{\delta}{1+Z_0\cdot\sin(2\delta)}\right)
\right]^{-1},
\label{eq_z_delta}
\end{align}
where $Z_0$ is defined via \eqref{eq_z0} in Lemma~\ref{lem_z0} above.
Then, 
\begin{align}
v(t) < f(t) \cdot Z_\delta,
\label{eq_z_delta_ineq}
\end{align}
for all $0 < t \leq s\left((i+1/2)\cdot\pi-\delta\right)$,
where the functions $f,v$ are defined, respectively, via \eqref{eq_jan_f},
\eqref{eq_jan_v}
in Section~\ref{sec_prufer}, and the function 
$s:\left[0,n\cdot\pi\right] \to \left[-x_n,x_n\right]$
is the inverse of $\theta$.
\label{lem_z_delta}
\end{lem}
\begin{proof}
Suppose that the point $t_\delta$ is defined via the formula
\begin{align}
t_\delta = s\left((i+1/2)\cdot\pi-\delta\right).
\label{eq_t_delta}
\end{align}
Due to \eqref{eq_dvf_positive} in the proof of Lemma~\ref{lem_theta_increasing},
the function $f/v$ 
decreases monotonically in the interval $(0,t_\delta)$, 
and therefore to prove \eqref{eq_z_delta_ineq} it suffices to show
that
\begin{align}
\frac{f(t_\delta)}{v(t_\delta)} > \frac{1}{Z_\delta} = 
1 + \frac{3}{2} \cdot \left(\frac{\pi}{4} + 
                            \frac{\delta}{1+Z_0\cdot\sin(2\delta)}\right).
\label{eq_fv_delta}
\end{align}
We observe that, due to Lemma~\ref{lem_theta_increasing},
\begin{align}
0 \leq \sin(2\theta(t)) \leq \sin(2\delta),
\label{eq_theta_delta}
\end{align}
for all $t_\delta \leq t \leq s((i+1/2)\pi)$. We combine 
\eqref{eq_t_delta},
\eqref{eq_theta_delta} with
Lemma~\ref{lem_prufer}, 
Lemma~\ref{lem_theta_increasing},
Lemma~\ref{lem_pi4},
Lemma~\ref{lem_dfv} and Lemma~\ref{lem_z0} to obtain
\begin{align}
& \frac{f(t_\delta)}{v(t_\delta)}-
\frac{f(s((i+1/2)\pi))}{v(s((i+1/2)\pi))} =
\nonumber \\
& \int_{t_\delta}^{s((i+1/2)\pi)} 
  \frac{d}{dt} \left(-\frac{f}{v}\right)(t) \; dt =
 \int_{(i-1/2)\pi-\delta}^{(i-1/2)\pi}
\frac{ \frac{d}{dt}\left(-\frac{f}{v}\right)(s(\eta)) \; d\eta}
{f(s(\eta)) + v(s(\eta)) \cdot \sin(2\eta) } >
\nonumber \\
& \int_{(i-1/2)\pi-\delta}^{(i-1/2)\pi}
\frac{ \frac{d}{dt}\left(-\frac{f}{v}\right)(s(\eta)) }{f(s(\eta))} \cdot
\frac{d\eta}
{1 + (v/f)(s(\eta)) \cdot \sin(2\delta) } > 
\nonumber \\
& \frac{3}{2} \cdot \delta \cdot \frac{1}{1 + Z_0 \cdot \sin(2\delta)}.
\label{eq_fv_delta_3}
\end{align}
We combine \eqref{eq_fv_delta_3} with \eqref{eq_fv_z0}
in Lemma~\ref{lem_z0} to obtain \eqref{eq_fv_delta}, which, in turn,
implies \eqref{eq_z_delta_ineq}.
\end{proof}
In the following two lemmas, we estimate the rate of decay of the ratio
$f/v$ and its relationship with $\theta$ of the ODE 
\eqref{eq_prufer_theta_ode} in Lemma~\ref{lem_prufer}.
\begin{lem}
Suppose that $n \geq 2$ and $(n+1)/2 \leq i \leq n-1$ are positive integers.
Suppose also that the real number $x_n$ and the function
$\theta:\left[-x_n,x_n\right] \to \Rc$ are those
of Lemma~\ref{lem_prufer}. Suppose furthermore that $0 < \delta < \pi/4$
is a real number. Then,
\begin{align}
\left(\frac{f}{v}\right)\left(s(i\pi-\delta)\right) -
\left(\frac{f}{v}\right)\left(s(i\pi-\delta+\pi/2)\right) >
2\cdot \sin(2\delta),
\label{eq_sin_delta_1}
\end{align}
where the functions $f,v$ are defined, respectively,
via \eqref{eq_jan_f}, \eqref{eq_jan_v}
in Section~\ref{sec_prufer}, and the function 
$s:\left[0,n\cdot\pi\right] \to \left[-x_n,x_n\right]$
is the inverse of $\theta$.
\label{lem_delta_small}
\end{lem}
\begin{proof}
We observe that, 
due to Lemma~\ref{lem_prufer} and Lemma~\ref{lem_theta_increasing},
\begin{align}
\sin(2\theta(t)) > 0,
\label{eq_sin_positive}
\end{align}
for all $s(i\pi) < t < s(i\pi-\delta+\pi/2)$.
We combine 
\eqref{eq_sin_positive} with
Lemma~\ref{lem_prufer}, 
Lemma~\ref{lem_theta_increasing},
Lemma~\ref{lem_pi4},
Lemma~\ref{lem_dfv} and Lemma~\ref{lem_z_delta} to obtain
\begin{align}
& \frac{f(s(i\pi))}{v(s(i\pi))}-
\frac{f(s(i\pi-\delta+\pi/2))}{v(s(i\pi-\delta+\pi/2))} =
\nonumber \\
& \int_{s(i\pi)}^{s(i\pi-\delta+\pi/2)} 
  \frac{d}{dt} \left(-\frac{f}{v}\right)(t) \; dt =
 \int_{i\pi}^{i\pi-\delta+\pi/2}
\frac{ \frac{d}{dt}\left(-\frac{f}{v}\right)(s(\eta)) \; d\eta}
{f(s(\eta)) + v(s(\eta)) \cdot \sin(2\eta) } >
\nonumber \\
& \int_{i\pi}^{i\pi-\delta+\pi/2}
\frac{ \frac{d}{dt}\left(-\frac{f}{v}\right)(s(\eta)) }{f(s(\eta))} \cdot
\frac{d\eta}
{1 + (v/f)(s(\eta)) } > 
 \frac{3}{2} \cdot \left(\frac{\pi}{2}-\delta \right) \cdot 
\frac{1}{1 + Z_\delta},
\label{eq_fv_delta_4}
\end{align}
where $Z_\delta$ is defined via \eqref{eq_z_delta}
in Lemma~\ref{lem_z_delta}. We also observe that,
due to Lemma~\ref{lem_prufer} and Lemma~\ref{lem_theta_increasing},
\begin{align}
\sin(2\theta(t)) < 0,
\label{eq_sin_negative}
\end{align}
for all $s(i\pi-\delta) < t < s(i\pi)$.
We combine 
\eqref{eq_sin_negative} with
Lemma~\ref{lem_prufer}, 
Lemma~\ref{lem_theta_increasing},
Lemma~\ref{lem_pi4} and
Lemma~\ref{lem_dfv} to obtain
\begin{align}
& \frac{f(s(i\pi-\delta))}{v(s(i\pi-\delta))}-
\frac{f(s(i\pi))}{v(s(i\pi))} =
\nonumber \\
& \int_{s(i\pi-\delta)}^{s(i\pi)} 
  \frac{d}{dt} \left(-\frac{f}{v}\right)(t) \; dt =
 \int_{i\pi-\delta}^{i\pi}
\frac{ \frac{d}{dt}\left(-\frac{f}{v}\right)(s(\eta)) \; d\eta}
{f(s(\eta)) + v(s(\eta)) \cdot \sin(2\eta) } >
\nonumber \\
& \int_{i\pi-\delta}^{i\pi}
\frac{ \frac{d}{dt}\left(-\frac{f}{v}\right)(s(\eta)) \; d\eta}{f(s(\eta))} > 
\frac{3}{2} \cdot \delta.
\label{eq_fv_delta_5}
\end{align}
Next, suppose that the function $h:\left[0,\pi/4\right] \to \Rc$
is defined via the formula
\begin{align}
h(\delta) = 
\frac{3}{2} \cdot \left(\frac{\pi}{2}-\delta \right) \cdot 
   \frac{1}{1 + Z_\delta} +
\frac{3}{2} \cdot \delta - 2 \cdot \sin(2\delta),
\label{eq_h_delta}
\end{align}
where $Z_\delta$ is defined via \eqref{eq_z_delta} in Lemma~\ref{lem_z_delta}.
One can easily verify that
\begin{align}
\min_\delta \left\{ h(\delta) \; : \; 0 \leq \delta \leq \pi/4 \right\}
> \frac{1}{25},
\label{eq_h_delta_min}
\end{align}
and, in particular, that $h(\delta) > 0$ for all $0 \leq \delta \leq \pi/4$.
We combine \eqref{eq_fv_delta_4}, \eqref{eq_fv_delta_5},
\eqref{eq_h_delta} and \eqref{eq_h_delta_min}
to obtain \eqref{eq_sin_delta_1}.
\end{proof}
\begin{lem}
Suppose that $n \geq 2$ and $(n+1)/2 \leq i \leq n-1$ are positive integers.
Suppose also that the real number $x_n$ and the function
$\theta:\left[-x_n,x_n\right] \to \Rc$ are those
of Lemma~\ref{lem_prufer}. Suppose furthermore that $0 < \delta < \pi/4$
is a real number. Then,
\begin{align}
\left(\frac{f}{v}\right)\left(s(i\pi+\delta-\pi/2)\right) -
\left(\frac{f}{v}\right)\left(s(i\pi+\delta)\right) >
2\cdot \sin(2\delta),
\label{eq_sin_delta_2}
\end{align}
where the functions $f,v$ are defined, respectively, 
via \eqref{eq_jan_f}, \eqref{eq_jan_v}
in Section~\ref{sec_prufer}, and the function 
$s:\left[0,n\cdot\pi\right] \to \left[-x_n,x_n\right]$
is the inverse of $\theta$.
\label{lem_delta_large}
\end{lem}
\begin{proof}
We observe that, 
due to Lemma~\ref{lem_prufer} and Lemma~\ref{lem_theta_increasing},
\begin{align}
\sin(2\theta(t)) > 0,
\label{eq_sin_positive2}
\end{align}
for all $s(i\pi) < t < s(i\pi+\delta)$.
We combine 
\eqref{eq_sin_positive2} with
Lemma~\ref{lem_prufer}, 
Lemma~\ref{lem_theta_increasing},
Lemma~\ref{lem_pi4},
Lemma~\ref{lem_dfv} and Lemma~\ref{lem_z_delta} to obtain
\begin{align}
& \frac{f(s(i\pi))}{v(s(i\pi))}-
\frac{f(s(i\pi+\delta))}{v(s(i\pi+\delta))} =
\nonumber \\
& \int_{s(i\pi)}^{s(i\pi+\delta)} 
  \frac{d}{dt} \left(-\frac{f}{v}\right)(t) \; dt =
 \int_{i\pi}^{i\pi+\delta}
\frac{ \frac{d}{dt}\left(-\frac{f}{v}\right)(s(\eta)) \; d\eta}
{f(s(\eta)) + v(s(\eta)) \cdot \sin(2\eta) } >
\nonumber \\
& \int_{i\pi}^{i\pi+\delta}
\frac{ \frac{d}{dt}\left(-\frac{f}{v}\right)(s(\eta)) }{f(s(\eta))} \cdot
\frac{d\eta}
{1 + (v/f)(s(\eta)) } > 
 \frac{3}{2} \cdot 
\frac{\delta}{1 + Z_\delta},
\label{eq_fv_delta_6}
\end{align}
where $Z_\delta$ is defined via \eqref{eq_z_delta}
in Lemma~\ref{lem_z_delta}. We also observe that,
due to Lemma~\ref{lem_prufer} and Lemma~\ref{lem_theta_increasing},
\begin{align}
\sin(2\theta(t)) < 0,
\label{eq_sin_negative2}
\end{align}
for all $s(i\pi+\delta-\pi/2) < t < s(i\pi)$.
We combine 
\eqref{eq_sin_negative2} with
Lemma~\ref{lem_prufer}, 
Lemma~\ref{lem_theta_increasing},
Lemma~\ref{lem_pi4} and
Lemma~\ref{lem_dfv} to obtain
\begin{align}
& \frac{f(s(i\pi+\delta-\pi/2))}{v(s(i\pi+\delta-\pi/2))}-
\frac{f(s(i\pi))}{v(s(i\pi))} =
\nonumber \\
& \int_{s(i\pi+\delta-\pi/2)}^{s(i\pi)} 
  \frac{d}{dt} \left(-\frac{f}{v}\right)(t) \; dt =
 \int_{i\pi+\delta-\pi/2}^{i\pi}
\frac{ \frac{d}{dt}\left(-\frac{f}{v}\right)(s(\eta)) \; d\eta}
{f(s(\eta)) + v(s(\eta)) \cdot \sin(2\eta) } >
\nonumber \\
& \int_{i\pi+\delta-\pi/2}^{i\pi}
\frac{ \frac{d}{dt}\left(-\frac{f}{v}\right)(s(\eta)) \; d\eta}{f(s(\eta))} > 
\frac{3}{2} \cdot \left(\frac{\pi}{2} - \delta\right).
\label{eq_fv_delta_7}
\end{align}
Obviously, for all $0 < \delta < \pi/4$,
\begin{align}
\frac{3}{2} \cdot \frac{\delta}{1 + Z_\delta} + 
\frac{3}{2} \cdot \left(\frac{\pi}{2} - \delta\right) >
\frac{3}{2} \cdot \left(\frac{\pi}{2}-\delta \right) \cdot 
   \frac{1}{1 + Z_\delta} +
\frac{3}{2} \cdot \delta.
\label{eq_flip}
\end{align}
We combine \eqref{eq_fv_delta_6},\eqref{eq_fv_delta_7},
\eqref{eq_flip} with \eqref{eq_h_delta}, \eqref{eq_h_delta_min}
in the proof of Lemma~\ref{lem_delta_small} to obtain
\eqref{eq_sin_delta_2}.
\end{proof}
In the following lemma, we analyze the integral of the oscillatory part
of the right-hand side of the ODE \eqref{eq_prufer_theta_ode}
between consecutive roots of $\psi_n$.
This lemma can be viewed as an extention of Lemma~\ref{lem_vsin_1},
and is used in the proof of Theorem~\ref{thm_n_lower} below.
\begin{lem}
Suppose that $n \geq 2$ is an integer,
$-1 < t_1 < t_2 < \dots < t_n < 1$ are the roots of $\psi_n$
in the interval $(-1,1)$,
and $x_1 < \dots < x_{n-1}$ are the roots of $\psi_n'$
in the interval $(t_1,t_n)$. Suppose also, that the real number $x_n$
and the function $\theta: \left[-x_n, x_n\right] \to \Rc$ are
those of Lemma~\ref{lem_prufer} above. Suppose furthermore
that the function $v$ is defined via \eqref{eq_jan_v}
in Section~\ref{sec_prufer}. Then,
\begin{align}
\int_{t_i}^{t_{i+1}} v(t) \cdot \sin(2\theta(t)) \; dt > 0,
\label{eq_int_t_tp}
\end{align}
for all integer $(n+1)/2 \leq i \leq n-1$, i.e. for all integer $i$
such that $0 \leq t_i < t_n$. 
\label{lem_vsin_2}
\end{lem}
\begin{proof}
Suppose that $i$ is a positive integer such that $(n-1)/2 \leq i \leq n-1$.
Suppose also that the function
$s: \left[0, n\cdot\pi\right] \to \left[-x_n,x_n\right]$ 
is the 
inverse of $\theta$. In other words, for all $0 \leq \eta \leq n \cdot \pi$,
\begin{align}
\theta(s(\eta)) = \eta.
\label{eq_s_eta_3}
\end{align}
Due to \eqref{eq_vsin_b} in the proof of Lemma~\ref{lem_vsin_1} above,
\begin{align}
& \int_{t_i}^{x_i} v(t) \cdot \sin(2\theta(t)) \; dt = \nonumber \\
& -\int_0^{\pi/2}
   \frac{v(s(i\pi+\eta-\pi/2)) \cdot \sin(2\eta) \; d\eta}
   {f(s(i\pi+\eta-\pi/2)) - v(s(i\pi+\eta-\pi/2)) \cdot \sin(2\eta)}.
\label{eq_vsin_d}
\end{align}
We proceed to compare the integrand in \eqref{eq_vsin_d} to 
the integrand in \eqref{eq_vsin_a} in the proof of
Lemma~\ref{lem_vsin_1}. First, for all $0 < \eta < \pi/4$,
\begin{align}
\frac{1}{(f/v)(s(i\pi+\eta-\pi/2)) - \sin(2\eta)} <
\frac{1}{(f/v)(s(i\pi+\eta)) + \sin(2\eta)},
\label{eq_up_to_pi4}
\end{align}
due to \eqref{eq_sin_delta_2} in Lemma~\ref{lem_delta_large}.
Moreover, for all $\pi/4 < \eta < \pi/2$, we substitute
$\delta = \pi/2-\eta$ to obtain
\begin{align}
& \frac{1}{(f/v)(s(i\pi+\eta-\pi/2)) - \sin(2\eta)} = 
  \frac{1}{(f/v)(s(i\pi-\delta))-\sin(2\delta)} < \nonumber \\
& \frac{1}{(f/v)(s(i\pi-\delta+\pi/2)) + \sin(2\delta)} =
\frac{1}{(f/v)(s(i\pi+\eta)) + \sin(2\eta)},
\label{eq_beyond_pi4}
\end{align}
due to \eqref{eq_sin_delta_1} in Lemma~\ref{lem_delta_small}.
We combine \eqref{eq_vsin_a} in the proof of
Lemma~\ref{lem_vsin_1} with \eqref{eq_vsin_d}, \eqref{eq_up_to_pi4},
\eqref{eq_beyond_pi4} to obtain \eqref{eq_int_t_tp}.
\end{proof}
The following theorem is a counterpart of Theorem~\ref{thm_n_upper}
above.
It is illustrated in
Tables~\ref{t:test77a}, \ref{t:test77b}, \ref{t:test77c}, \ref{t:test99}.
\begin{thm}
Suppose that $n \geq 2$ is a positive integer. Suppose also that
$t_n$
is the maximal root of $\psi_n$ in $(-1,1)$. Then,
\begin{align}
1 + 
\frac{2}{\pi} \int_0^{t_n} \sqrt{ \frac{\chi_n - c^2 t^2}{1 - t^2} } \; dt <
n.
\label{eq_n_lower}
\end{align}
\label{thm_n_lower}
\end{thm}
\begin{proof}
Suppose that the real numbers
\begin{align}
-1 \leq -x_n < t_1 < x_1 < t_2 < \dots < t_{n-1} < x_{n-1} < t_n < x_n \leq 1
\label{eq_all_xt_2}
\end{align}
and the function $\theta: \left[-x_n,x_n\right] \to \Rc$ are 
those of Lemma~\ref{lem_prufer} above. Suppose also that
the functions $f,v$ are defined, respectively, via 
\eqref{eq_jan_f},\eqref{eq_jan_v} in Section~\ref{sec_prufer}.
If $n$ is odd, then we combine
\eqref{eq_prufer_theta_ode}, \eqref{eq_theta_at_t},
in Lemma~\ref{lem_prufer} with
\eqref{eq_int_t_tp} in Lemma~\ref{lem_vsin_2} to obtain
\begin{align}
\frac{n-1}{2} \cdot \pi & \; =
\int_{t_{(n+1)/2}}^{t_n} \theta'(t) \; dt =
\int_0^{t_n} f(t) \; dt +
\sum_{i=(n+1)/2}^{n-1} \int_{t_i}^{t_{i+1}} v(t) \cdot \sin(2\theta(t)) \; dt
\nonumber \\
& \; > \int_0^{t_n} f(t) \; dt.
\label{eq_n_lower_odd}
\end{align}
If $n$ is even, then we combine
\eqref{eq_prufer_theta_ode}, \eqref{eq_theta_at_t},
\eqref{eq_theta_at_x} in Lemma~\ref{lem_prufer} with
\eqref{eq_int_x_tp} in Lemma~\ref{lem_vsin_1} 
and
\eqref{eq_int_t_tp} in Lemma~\ref{lem_vsin_2} to obtain
\begin{align}
\frac{n-1}{2} \cdot \pi & \; = 
\int_{x_{n/2}}^{t_n} \theta'(t) \; dt
= 
\int_0^{t_n} f(t) \; dt \; + \nonumber \\
& \; \; \; \; \; 
\int_{x_{n/2}}^{t_{(n/2)+1}} v(t) \cdot \sin(2\theta(t)) \; dt +
\sum_{i=(n/2)+1}^{n-1} \int_{t_i}^{t_{i+1}} v(t) \cdot \sin(2\theta(t)) \; dt
\nonumber \\
& \; > \int_0^{t_n} f(t) \; dt.
\label{eq_n_lower_even}
\end{align}
We combine \eqref{eq_n_lower_odd} and \eqref{eq_n_lower_even}
to conclude \eqref{eq_n_lower}.
\end{proof}
\begin{cor}
Suppose that $n \geq 2$ is a positive integer, 
and that $\chi_n > c^2$. Suppose also that 
$t_n$ is the maximal root of $\psi_n$
in the interval $(-1, 1)$. Then,
\begin{align}
1 + \frac{2}{\pi} \sqrt{\chi_n} \cdot
    E \brk{ \text{asin}\brk{t_n}, \frac{c}{\sqrt{\chi_n}} }
< n < 
\frac{2}{\pi} \sqrt{\chi_n} \cdot E \brk{ \frac{c}{\sqrt{\chi_n}} },
\label{eq_jan_n_both_with_e}
\end{align}
where $E(y, k)$ and $E(k)$ are defined, respectively, via \eqref{eq_E_y}
and \eqref{eq_E} in Section~\ref{sec_elliptic}.
\end{cor}
\begin{proof}
It follows immediately from \eqref{eq_E_y_2},
\eqref{eq_n_upper_large} in Theorem~\ref{thm_n_upper} and
\eqref{eq_n_lower} in Theorem~\ref{thm_n_lower}.
\end{proof}
The following theorem,
illustrated in
Tables~\ref{t:test98a}, \ref{t:test98b},
provides upper and lower bounds 
on the distance between consecutive roots of $\psi_n$ inside $\brk{-1, 1}$.
\begin{thm}
Suppose that $n \geq 2$ is a positive integer,
and that $\chi_n > c^2$. Suppose also that
$-1 < t_1 < t_2 < \dots < t_n < 1$ are the roots of $\psi_n$
in the interval $(-1,1)$. Suppose furthermore
that the functions $f,v$ are defined,
respectively, via \eqref{eq_jan_f},\eqref{eq_jan_v}
in Section~\ref{sec_prufer}.
Then, 
\begin{align}
\frac{\pi}{ f(t_{i+1}) + v(t_{i+1})/2 } <
t_{i+1} - t_{i} <
\frac{\pi}{ f(t_i) },
\label{eq_zeros_inside_bounds}
\end{align}
for all integer $(n+1)/2 \leq i \leq n-1$, i.e. for all integer $i$
such that $0 \leq t_i < t_n$. 
\label{thm_zeros_inside_bounds}
\end{thm}
\begin{proof}
Suppose that the function $\theta: \left[-1,1\right] \to \Rc$ is
that of Lemma~\ref{lem_prufer}. We observe that $f$
is increasing in $(0,1)$ due to \eqref{eq_jan_f} in Section~\ref{sec_prufer},
and combine this observation with
\eqref{eq_prufer_theta_ode}, \eqref{eq_theta_at_t},
in Lemma~\ref{lem_prufer} and
\eqref{eq_int_t_tp} in Lemma~\ref{lem_vsin_2} to obtain
\begin{align}
\pi & \; =
\int_{t_i}^{t_{i+1}} \theta'(t) \; dt =
\int_{t_i}^{t_{i+1}} f(t) \; dt +
\int_{t_i}^{t_{i+1}} v(t) \cdot \sin(2\theta(t)) \; dt \nonumber \\
& \; >
\int_{t_i}^{t_{i+1}} f(t) \; dt >
\left(t_{i+1} - t_i\right) \cdot f(t_i),
\end{align}
which implies the right-hand side of \eqref{eq_zeros_inside_bounds}.
As in Lemma~\ref{lem_prufer}, 
suppose that $x_i$ is the zero of $\psi_n'$ in the interval
$(t_i, t_{i+1})$.
We combine
\eqref{eq_prufer_theta_ode}, \eqref{eq_theta_at_t}, \eqref{eq_theta_at_t}
in Lemma~\ref{lem_prufer} and \eqref{eq_int_x_tp},
\eqref{eq_int_tp_xp} in Lemma~\ref{lem_vsin_1}
to obtain
\begin{align}
\int_{t_i}^{x_i} f(t) \; dt > 
\int_{t_i}^{x_i} \theta'(t) \; dt = \pi =
\int_{x_i}^{t_{i+1}} \theta'(t) \; dt > 
\int_{x_i}^{t_{i+1}} f(t) \; dt.
\label{eq_t_x_tp}
\end{align}
Since $f$ is increasing in $(t_i,t_{i+1})$ due to \eqref{eq_jan_f}
in Section~\ref{sec_prufer},
the inequality \eqref{eq_t_x_tp} implies that
\begin{align}
x_i - t_i > t_{i+1} - x_i.
\label{eq_x_t_tp_x}
\end{align}
Moreover, we observe that $v$ is also increasing in $(0,1)$.
We combine this observation with \eqref{eq_x_t_tp_x},
\eqref{eq_prufer_theta_ode}, \eqref{eq_theta_at_t},
in Lemma~\ref{lem_prufer} and
\eqref{eq_int_x_tp}, \eqref{eq_int_tp_xp} in Lemma~\ref{lem_vsin_1} to obtain
\begin{align}
\pi & \; =
\int_{t_i}^{t_{i+1}} \theta'(t) \; dt <
\int_{t_i}^{t_{i+1}} f(t) \; dt +
\int_{x_i}^{t_{i+1}} v(t) \cdot \sin(2\theta(t)) \; dt \nonumber \\
& \; <
\left(t_{i+1}-t_i\right) \cdot f(t_{i+1}) +
\left(t_{i+1}-x_i\right) \cdot v(t_{i+1}) \nonumber \\
& \; < \left(t_{i+1}-t_i\right) \cdot f(t_{i+1}) +
\frac{t_{i+1}-t_i}{2} \cdot v(t_{i+1}),
\end{align}
which implies the left-hand side of \eqref{eq_zeros_inside_bounds}.
\end{proof}
The following theorem is a direct consequence of
Theorem~\ref{thm_n_upper} above.
\begin{thm}
Suppose that $n \geq 2$ is a positive integer.
If $n \geq 2c/\pi$, then $\chi_n > c^2$.
\label{thm_good_n_good_khi}
\end{thm}
\begin{proof}
Suppose that $\chi_n < c^2$, and $T$ is the maximal root of $\psi'_n$
in $(0,1)$, as in 
Theorem~\ref{thm_n_upper} above.
Then,
\begin{align}
n & \; < \frac{2}{\pi} 
         \int_0^{T} \sqrt{ \frac{\chi_n - c^2 t^2}{1 - t^2} } \; dt
    = \frac{2c}{\pi} \int_0^{T} \sqrt{ \frac{\chi_n/c^2 - t^2}{1 - t^2} } \; dt
< \frac{ 2 c }{ \pi }\cdot T  < \frac{ 2c }{ \pi },
\end{align}
due to \eqref{eq_n_upper_small} in Theorem~\ref{thm_n_upper}.
\end{proof}
\subsubsection{A Certain Transformation of the Prolate ODE}
\label{sec_first_order}
In this subsection, we analyze the oscillation properties of $\psi_n$
via 
transforming the ODE \eqref{eq_prolate_ode} into
a second-order linear ODE without the first-order term.
The following lemma is the principal technical tool of this subsection.
\begin{lem}
Suppose that $n \geq 0$ is a non-negative integer. Suppose also that
that the functions $\Psi_n, Q_n: (-1,1) \to \Rc$
are defined, respectively, via the formulae
\begin{align}
\Psi_n(t) = \psi_n(t) \cdot \sqrt{1-t^2}
\label{eq_big_psi_n}
\end{align}
and
\begin{align}
Q_n(t) = 
\frac{ \chi_n - c^2\cdot t^2 }{ 1 - t^2 } +
       \frac{ 1 }{ \brk{1 - t^2}^2 },
\label{eq_big_q_n}
\end{align}
for $-1 < t < 1$. Then,
\begin{align}
\Psi_n''(t) + Q_n(t) \cdot \Psi_n(t) = 0,
\label{eq_big_psi_ode}
\end{align}
for all $-1 < t < 1$.
\label{lem_trans}
\end{lem}
\begin{proof}
We differentiate $\Psi_n$ with respect to $t$ to obtain
\begin{align}
\Psi_n'(t) = \psi_n'(t) \sqrt{1-t^2} -
             \psi_n(t) \cdot \frac{t}{\sqrt{1-t^2}}.
\label{eq_d_big_psi}
\end{align}
Then, using \eqref{eq_d_big_psi}, we differentiate $\Psi_n'$ with 
respect to $t$ to obtain
\begin{align}
\Psi_n''(t)
& \; = 
\psi_n''(t) \sqrt{1-t^2} -
\psi_n'(t) \cdot \frac{ 2t }{ \sqrt{t^2 - 1} } -
\psi_n(t) \cdot
   \frac{ \sqrt{1-t^2} + t^2 / \sqrt{1-t^2} }{ 1-t^2 }
\nonumber \\
& \; =
\psi_n''(t) \sqrt{1-t^2} - 
\psi_n'(t) \cdot \frac{ 2t }{ \sqrt{1-t^2} } -
\psi_n(t) \brk{1-t^2}^{-\frac{3}{2}}
\nonumber \\
& \; = 
\frac{1}{\sqrt{1-t^2}}
\sbrk{
  \brk{1-t^2} \cdot \psi_n''(t) - 2t \cdot \psi_n'(t) - 
  \frac{ \psi_n(t) }{ 1-t^2 }
}
\nonumber \\
& \; = 
\frac{1}{\sqrt{1-t^2}}
\sbrk{ -\psi_n(t) \cdot \left(\chi_n - c^2 \cdot t^2 \right) -
       \frac{ \psi_n(t) }{ 1-t^2 } }
\nonumber \\
& \; = 
-\Psi_n(t) \cdot
\brk{ \frac{ \chi_n - c^2 \cdot t^2 }{1-t^2} + \frac{1}{\brk{t^2-1}^2}}.
\label{eq_dd_big_psi}
\end{align}
We observe that \eqref{eq_big_psi_ode} follows from \eqref{eq_dd_big_psi}.
\end{proof}

In the next theorem, we provide an upper bound on $\chi_n$
in terms of $n$.
The results of the corresponding numerical experiments
are reported in
Tables~\ref{t:test80a}, \ref{t:test80b}.
\begin{thm}
Suppose that $n \geq 2$ is a positive integer, and that $\chi_n > c^2$. Then,
\begin{align}
\chi_n < \brk{ \frac{\pi}{2} \brk{n+1} }^2.
\label{eq_khi_n_square}
\end{align}
\label{thm_khi_n_square}
\end{thm}
\begin{proof}
Suppose that the functions $\Psi_n, Q_n: (-1,1) \to \Rc$
are those of Lemma~\ref{lem_trans} above.
We observe that, since $\chi_n > c^2$,
\begin{align}
Q_n(t) > \chi_n + 1,
\label{eq_qn_gt_khi}
\end{align}
for $-1 < t < 1$. Suppose now that $t_n$ is the maximal root
of $\psi_n$ in $(-1,1)$. We combine \eqref{eq_qn_gt_khi} with
\eqref{eq_big_psi_ode} in Lemma~\ref{lem_trans} above and
Theorem~\ref{thm_08_12_zeros},
Corollary~\ref{cor_08_12_zeros} in Section~\ref{sec_oscillation_ode}
to obtain the inequality
\begin{align}
t_n \geq 1 - \frac{ \pi }{ \sqrt{\chi_n + 1} }.
\label{eq_tn_khi}
\end{align}
Then, we combine \eqref{eq_tn_khi} with 
Theorem~\ref{thm_n_lower} above to obtain
\begin{align}
n
& \; > 1 + \frac{2}{\pi}
   \int_0^{t_n} \sqrt{ \frac{\chi_n - c^2 t^2}{1 - t^2} } \; dt
\nonumber \\
& \; > 1 + \frac{2 \cdot t_n}{\pi} \sqrt{\chi_n} \geq
       1 + \frac{2}{\pi} \sqrt{\chi_n} \brk{1 - \frac{\pi}{\sqrt{\chi_n+1}}}
     > \frac{2}{\pi} \sqrt{\chi_n} - 1,
\end{align}
which implies \eqref{eq_khi_n_square}.
\end{proof}
The following theorem is a consequence of the proof
of Theorem~\ref{thm_khi_n_square}.
\begin{thm}
Suppose that $n \geq 2$ is a positive integer, and
that $\chi_n > c^2$.
Suppose also that $t_1 < \dots <t_n$ are
the roots of $\psi_n$ in $(-1,1)$. Then, 
\begin{align}
t_{j+1} - t_j < \frac{\pi}{\sqrt{\chi_n+1}},
\label{eq_tspacing}
\end{align}
for all $j=1,2,\dots,n-1$.
\label{thm_spacing_khi}
\end{thm}
\begin{proof}
The inequality \eqref{eq_tspacing}
follows from the combination
of 
\eqref{eq_qn_gt_khi} in the proof of Theorem~\ref{thm_khi_n_square},
\eqref{eq_big_psi_ode} in Lemma~\ref{lem_trans} and
Theorem~\ref{thm_08_12_zeros},
Corollary~\ref{cor_08_12_zeros} in Section~\ref{sec_oscillation_ode}.
\end{proof}
The following theorem extends 
Theorem~\ref{thm_good_n_good_khi} in Section~\ref{sec_sharp}.
\begin{thm}
Suppose that $n \geq 2$ is a positive integer.
\begin{itemize}
\item If $n \leq (2c/\pi)-1$, then $\chi_n < c^2$.
\item If $n \geq (2c/\pi)$, then $\chi_n > c^2$.
\item If $(2c/\pi)-1 < n < (2c/\pi)$, then either inequality is possible.
\end{itemize}
\label{thm_n_and_khi}
\end{thm}
\begin{proof}
Suppose that $\chi_n > c^2$, 
and that the functions $\Psi_n, Q_n: (-1,1) \to \Rc$
are those of Lemma~\ref{lem_trans} above.
Suppose also that $t_1 < \dots < t_n$ are the roots
of $\psi_n$ in $(-1,1)$. We observe that, due to \eqref{eq_big_q_n}
in Lemma~\ref{lem_trans},
\begin{align}
Q_n(t) = 
c^2 + \frac{ \chi_n - c^2 }{ 1 - t^2 } + \frac{ 1 }{ \brk{1-t^2}^2 }
> c^2.
\label{eq_qn_gt_c}
\end{align}
We combine \eqref{eq_qn_gt_c} with
\eqref{eq_big_psi_ode} in Lemma~\ref{lem_trans} above and
Theorem~\ref{thm_25_09} in Section~\ref{sec_oscillation_ode} to conclude
that
\begin{align}
t_{j+1} - t_j < \frac{ \pi }{ c }, 
\label{eq_spacing_c}
\end{align}
for all $j = 1, \dots, n-1$, and, moreover,
\begin{align}
\quad 1 - t_n < \frac{ \pi }{ c }.
\label{eq_1_tn_c}
\end{align}
We combine \eqref{eq_spacing_c} with \eqref{eq_1_tn_c} to obtain
the inequality
\begin{align}
2 \brk{1 - \frac{ \pi }{ c }} < 2t_n = t_n - t_1 < \brk{n-1} \frac{ \pi }{ c },
\end{align}
which implies that
\begin{align}
n > \frac{2}{\pi} c - 1.
\label{eq_misc_n_small}
\end{align}
We conclude the proof by combining 
Theorem~\ref{thm_good_n_good_khi} in Section~\ref{sec_sharp}
with \eqref{eq_misc_n_small}.
\end{proof}
The following theorem is yet another application
of Lemma~\ref{lem_trans} above.
\begin{thm}
Suppose that $n \geq 2$ is a positive integer.
Suppose also that
$-1 < t_1 < t_2 < \dots < t_n < 1$ are the roots of $\psi_n$
in the interval $(-1,1)$. Suppose furthermore that
$i$ is an integer such that
$0 \leq t_i < t_n$, i.e. $(n+1)/2 \leq i \leq n-1$.
If $\chi_n > c^2$, then
\begin{align}
t_{i+1} - t_i > t_{i+2} - t_{i+1} > \dots > t_n - t_{n-1}.
\label{eq_spacing_shrink}
\end{align}
If $\chi_n < c^2 - c\sqrt{2}$, then
\begin{align}
t_{i+1} - t_i < t_{i+2} - t_{i+1} < \dots < t_n - t_{n-1}.
\label{eq_spacing_stretch}
\end{align}
\label{thm_spacing_inside}
\end{thm}
\begin{proof}
Suppose that the functions $\Psi_n, Q_n: (-1,1) \to \Rc$
are those of Lemma~\ref{lem_trans} above.
If $\chi_n > c^2$, then, due to \eqref{eq_big_q_n} in Lemma~\ref{lem_trans},
\begin{align}
Q_n(t) = c^2 + \frac{\chi_n - c^2}{1 - t^2} + \frac{1}{\brk{1-t^2}^2}
\label{eq_30_11_q_again}
\end{align}
is obviously a monotonically increasing function. We combine this
observation
with \eqref{eq_big_psi_ode} of Lemma~\ref{lem_trans} and
\eqref{eq_shrink}
of Theorem~\ref{thm_25_09_2} in Section~\ref{sec_oscillation_ode}
to conclude \eqref{eq_spacing_shrink}.

Suppose now that 
\begin{align}
\chi_n < c^2 - c \sqrt{2}. 
\label{eq_khi_very_small}
\end{align}
Suppose also that
the function $P_n : (1,\infty) \to \Rc$ is defined via the formula
\begin{align}
P_n(y) = Q_n\left( \sqrt{1-\frac{1}{\sqrt{y}}} \right) =
        y^2 + (\chi_n-c^2) \cdot y + c^2,
\label{eq_big_p_n}
\end{align}
for $1 < y < \infty$. 
Obviously,
\begin{align}
Q_n(t) = P_n\left(\frac{1}{1-t^2}\right).
\label{eq_qn_pn}
\end{align}
Suppose also that $y_0$ is defined via 
the formula
\begin{align}
y_0 = \frac{1}{1-(\sqrt{\chi_n}/c)^2} = \frac{c^2}{c^2 - \chi_n}.
\label{eq_y0}
\end{align}
We combine \eqref{eq_khi_very_small}, \eqref{eq_big_p_n} and \eqref{eq_y0} to
conclude that, for $1 < y < y_0$,
\begin{align}
P_n'(y) = 2y-(c^2-\chi_n) < 2y_0 - (c^2-\chi_n)
        = \frac{2c^2-(c^2-\chi_n)^2}{c^2-\chi_n} < 0.
\label{eq_dp_neg}
\end{align}
Moreover, due to \eqref{eq_big_p_n}, \eqref{eq_y0}, \eqref{eq_dp_neg},
\begin{align}
P_n(y) > P_n(y_0) = \left( \frac{c^2}{\chi_n-c^2} \right)^2 > 0,
\label{eq_pn_pos}
\end{align}
for all $1 < y < y_0$. We combine \eqref{eq_big_p_n}, \eqref{eq_qn_pn},
\eqref{eq_y0}, \eqref{eq_dp_neg} and \eqref{eq_pn_pos} to conclude
that $Q_n$ is monotonically decreasing and strictly positive 
in the interval $(0, \sqrt{\chi_n}/c)$.
We combine this observation with
\eqref{eq_stretch} of 
Theorem~\ref{thm_25_09_2} in Section~\ref{sec_oscillation_ode},
\eqref{eq_khi_small} of Lemma~\ref{lem_five}, and
\eqref{eq_big_psi_ode} of Lemma~\ref{lem_trans} to conclude
\eqref{eq_spacing_stretch}.
\end{proof}
\begin{remark}
Numerical experiments confirm that there exist real $c>0$ and 
integer $n>0$ such that
$c^2 - c\sqrt{2} < \chi_n < c^2$ and neither of
\eqref{eq_spacing_shrink}, \eqref{eq_spacing_stretch}
is true.
\label{rem_spacing_inside}
\end{remark}
In the following theorem, we provide
an upper bound on $1-t_n$, where
$t_n$ is the maximal root of $\psi_n$ in the interval $(-1,1)$.
\begin{thm}
Suppose that $n \geq 2$ is a positive integer,
and that $\chi_n > c^2$. Suppose also that $t_n$ is the maximal root
of $\psi_n$ in the interval $(-1,1)$. 
Then,
\begin{align}
c^2 \cdot (1-t_n)^2 + \frac{\chi_n-c^2}{1+t_n} \cdot (1-t_n) < \pi^2.
\label{eq_tn_pi}
\end{align}
Moreover,
\begin{align}
1-t_n < \frac{ 4\pi^2 }{\chi_n-c^2 + \sqrt{(\chi_n-c^2)^2+(4\pi c)^2} }.
\label{eq_tn_upper}
\end{align}
\label{thm_tn_upper}
\end{thm}
\begin{proof}
Suppose that the functions $\Psi_n, Q_n: (-1,1) \to \Rc$
are those of Lemma~\ref{lem_trans} above.
Since $\chi_n > c^2$, the function
$Q_n$ is monotonically increasing, i.e.
\begin{align}
Q_n(t_n) \leq Q(t),
\label{eq_qn_ineq_1}
\end{align}
for all $t_n \leq t < 1$.
We consider the solution $\varphi_n$ of the ODE
\begin{align}
\varphi_n''(t) + Q_n(t_n) \cdot \varphi_n(t) = 0,
\label{eq_phi_ode_1}
\end{align}
with the initial conditions
\begin{align}
\varphi(t_n) = \Psi_n(t_n) = 0, \quad \varphi'(t_n) = \Psi_n'(t_n).
\label{eq_phi_ic}
\end{align}
The function $\varphi_n$ has a root $y_n$ given via the formula
\begin{align}
y_n = t_n + \frac{ \pi }{ \sqrt{Q_n(t_n)} }.
\end{align}
Suppose, by contradiction, that $y_n \leq 1$.
Then, due to the combination of
\eqref{eq_big_psi_ode} of Lemma~\ref{lem_trans},
Theorem~\ref{thm_08_12_zeros}, Corollary~\ref{cor_08_12_zeros} in
Section~\ref{sec_oscillation_ode}, and \eqref{eq_qn_ineq_1} above,
$\Psi_n$ has a root in the interval $(t_n,y_n)$, in contradiction
to \eqref{eq_big_psi_n}. Therefore,
\begin{align}
t_n + \frac{\pi}{\sqrt{Q_n(t_n)}} > 1.
\label{eq_tn_qn_1}
\end{align}
We rewrite \eqref{eq_tn_qn_1} as
\begin{align}
(1-t_n)^2 \cdot Q_n(t_n) < \pi^2,
\label{eq_tn_qn_2}
\end{align}
and plug \eqref{eq_big_q_n} into \eqref{eq_tn_qn_2} to obtain
the inequality
\begin{align}
c^2 \cdot (1-t_n)^2 + \frac{\chi_n-c^2}{1+t_n} \cdot (1-t_n) +
\frac{1}{(1+t_n)^2} < \pi^2,
\label{eq_tn_qn_3}
\end{align}
which immediately yields \eqref{eq_tn_pi}. Since $1-t_n$ is positive,
\eqref{eq_tn_qn_3} implies that $1-t_n$
is bounded from above by the maximal root 
$x_{\max}$ of the quadratic equation
\begin{align}
c^2 \cdot x^2 + \frac{\chi_n-c^2}{2} \cdot x - \pi^2 = 0,
\label{eq_tn_quad_1}
\end{align}
given via the formula
\begin{align}
x_{\max} & \; =
\frac{1}{4c^2} \cdot \left(
\sqrt{ (\chi_n-c^2)^2 + 16 \pi^2 c^2 } - (\chi_n-c^2) \right) \nonumber \\
& \; = \frac{16 \pi^2 c^2}{4c^2} \cdot
\frac{1}{ \chi_n-c^2 + \sqrt{ (\chi_n-c^2)^2 + 16 \pi^2 c^2 } },
\end{align}
which implies \eqref{eq_tn_upper}.
\end{proof}
The following theorem uses Theorem~\ref{thm_tn_upper} to simplify
the inequalities
\eqref{eq_n_upper_large} in Theorem~\ref{thm_n_upper}
and \eqref{eq_n_lower} in Theorem~\ref{thm_n_lower}
in Section~\ref{sec_sharp}.
\begin{thm}
Suppose that $n \geq 2$ is a positive integer, and
that $\chi_n > c^2$. Suppose also that $t_n$ is the maximal root
of $\psi_n$ in the interval $(-1,1)$. Then,
\begin{align}
\frac{2}{\pi} \int_{t_n}^1
\sqrt{ \frac{\chi_n - c^2 t^2}{1-t^2} } \; dt < 4.
\label{eq_bound_4}
\end{align}
Moreover,
\begin{align}
n <
\frac{2}{\pi} \int_0^1 \sqrt{ \frac{\chi_n - c^2 t^2}{1 - t^2} } \; dt
< n+3.
\label{eq_both_large_simple}
\end{align}
\label{thm_n_khi_simple}
\end{thm}
\begin{proof}
We observe that, for $t_n \leq t < 1$,
\begin{align}
\frac{\chi_n-c^2 t^2}{1+t} < 
\frac{\chi_n-c^2 t_n^2}{1+t_n} = 
\frac{\chi_n-c^2}{1+t_n} + \frac{c^2-c^t t_n^2}{1+t_n} =
c^2 (1-t_n) + \frac{\chi_n-c^2}{1+t_n}.
\label{eq_simple_1}
\end{align}
We combine \eqref{eq_simple_1} with \eqref{eq_tn_pi} 
in Theorem~\ref{thm_tn_upper} to obtain the inequality
\begin{align}
\frac{\chi_n-c^2 t^2}{1+t} < 
\frac{\pi^2}{1-t_n},
\label{eq_simple_2}
\end{align}
valid for $t_n \leq t < 1$. We conclude from \eqref{eq_simple_2} that
\begin{align}
\int_{t_n}^1 \sqrt{ \frac{\chi_n - c^2 t^2}{1 - t^2} } \; dt <
\frac{\pi}{\sqrt{1-t_n}} \cdot \int_{t_n}^1 \frac{dt}{\sqrt{1-t}} =
\frac{\pi}{\sqrt{1-t_n}} \cdot 2\sqrt{1-t_n} = 2\pi,
\end{align}
which implies \eqref{eq_bound_4}. The inequality
\eqref{eq_both_large_simple} follows from the combination of
\eqref{eq_bound_4},
\eqref{eq_n_upper_large} in Theorem~\ref{thm_n_upper}
and \eqref{eq_n_lower} in Theorem~\ref{thm_n_lower}
in Section~\ref{sec_sharp}.
\end{proof}
\begin{cor}
Suppose that $n \geq 2$ is a positive integer, and that $\chi_n > c^2$.
Then,
\begin{align}
n < 
\frac{2}{\pi} \sqrt{\chi_n} \cdot E \brk{ \frac{c}{\sqrt{\chi_n}} } < n+3,
\label{eq_khi_simple_elliptic}
\end{align}
where $E(k)$ is defined via \eqref{eq_E} in Section~\ref{sec_elliptic}.
\label{cor_khi_simple}
\end{cor}
\begin{proof}
The inequality \eqref{eq_khi_simple_elliptic}
follows immediately from the combination of 
\eqref{eq_E} in Section~\ref{sec_elliptic}
and
\eqref{eq_both_large_simple} in Theorem~\ref{thm_n_khi_simple} above.
\end{proof}
The following theorem
extends Theorem~\ref{thm_tn_upper} above by providing
a lower bound on $1-t_n$, where
$t_n$ is the maximal root of $\psi_n$ in the interval $(-1,1)$.
\begin{thm}
Suppose that $n \geq 2$ is a positive integer, and
that $\chi_n > c^2$. Suppose also that $t_n$ is the maximal root
of $\psi_n$ in the interval $(-1,1)$. 
Then,
\begin{align}
\frac{ \pi^2/8 }{\chi_n-c^2 + \sqrt{(\chi_n-c^2)^2+(\pi c/2)^2} } < 1-t_n.
\label{eq_tn_lower}
\end{align}
\label{thm_tn_lower}
\end{thm}
\begin{proof}
We combine the inequalities
\eqref{eq_n_upper_large} in Theorem~\ref{thm_n_upper}
and \eqref{eq_n_lower} in Theorem~\ref{thm_n_lower}
in Section~\ref{sec_sharp} to conclude that
\begin{align}
1 < \frac{2}{\pi} \int_{t_n}^1 \sqrt{\frac{\chi_n-c^2t^2}{1-t^2}} \; dt.
\label{eq_lower_1}
\end{align}
We combine \eqref{eq_lower_1} with \eqref{eq_simple_1}
in the proof of Theorem~\ref{thm_n_khi_simple} above to obtain
\begin{align}
1 & \; < \frac{2}{\pi} \sqrt{
  \frac{\chi_n-c^2}{1+t_n} + c^2(1-t_n)
} \cdot \int_{t_n}^1 \frac{dt}{\sqrt{1-t}} \nonumber \\
& \; <
\frac{4}{\pi} \sqrt{ c^2 (1-t_n)^2 + (\chi_n-c^2) \cdot (1-t_n) }.
\label{eq_lower_2}
\end{align}
We rewrite \eqref{eq_lower_2} as
\begin{align}
c^2 (1-t_n)^2 + (\chi_n-c^2) \cdot (1-t_n) - \frac{\pi^2}{16} > 0.
\label{eq_lower_3}
\end{align}
Since $1-t_n$ is positive,
\eqref{eq_lower_3} implies that $1-t_n$
it is bounded from below by the maximal root 
$x_{\max}$ of the quadratic equation
\begin{align}
c^2 \cdot x^2 + \frac{\chi_n-c^2}{2} \cdot x - \frac{\pi^2}{16} = 0,
\label{eq_lower_4}
\end{align}
given via the formula
\begin{align}
x_{\max} & \; =
\frac{1}{2c^2} \cdot \left(
\sqrt{ (\chi_n-c^2)^2 + \pi^2 c^2/4 } - (\chi_n-c^2) \right) \nonumber \\
& \; = \frac{\pi^2 c^2}{8c^2} \cdot
\frac{1}{ \chi_n-c^2 + \sqrt{ (\chi_n-c^2)^2 + \pi^2 c^2/4 } },
\end{align}
which implies \eqref{eq_tn_lower}.
\end{proof}
The following theorem is a direct consequence of
Theorem~\ref{thm_n_khi_simple}. It is illustrated
in Figures~\ref{fig:test171a},~\ref{fig:test171b}.
\begin{thm}
Suppose that $n \geq 2$ is a positive integer such that
$n > 2c/\pi$, and that the function
$f:[0,\infty) \to \Rc$ is defined via the formula
\begin{align}
f(x) = -1 + \int_0^{\pi/2} \sqrt{ x + \cos^2(\theta)} \; d\theta.
\label{eq_f_def}
\end{align}
Suppose also that the function $H: [0,\infty) \to \Rc$ is the inverse of $f$,
in other words, 
\begin{align}
y = f(H(y)) = 
-1 + \int_0^{\pi/2} \sqrt{ H(y) + \cos^2(\theta)} \; d\theta,
\label{eq_big_h_def}
\end{align}
for all real $y \geq 0$.
Then,
\begin{align}
H\left( \frac{n\pi}{2c} - 1 \right) < 
\frac{\chi_n - c^2}{c^2} < 
H\left( \frac{n\pi}{2c} - 1 + \frac{3\pi}{2c} \right).
\label{eq_khi_via_h}
\end{align}
\label{thm_exp_term}
\end{thm}
\begin{proof}
Obviously, the function $f$, defined via \eqref{eq_f_def},
is monotonically increasing. Moreover, $f(0) = 0$, and
\begin{align}
\lim_{x \to \infty} f(x) = \infty.
\end{align}
Therefore, $H(y)$ 
is well defined for all $y \geq 0$, and, moreover,
the function $H$
is monotonically increasing. This observation,
combined with Theorems~\ref{thm_n_and_khi},
\ref{thm_n_khi_simple} above,
implies
the inequality \eqref{eq_khi_via_h}.
\end{proof}
In the following theorem, we provide a simple lower bound on $H$, 
defined via \eqref{eq_big_h_def}
in Theorem~\ref{thm_exp_term}.
\begin{thm}
Suppose that the function $H:[0,\infty) \to \Rc$ is defined
via \eqref{eq_big_h_def} in Theorem~\ref{thm_exp_term}.
Then,
\begin{align}
s \leq H\left( \frac{s}{4} \cdot \log \frac{16e}{s} \right),
\label{eq_h_bounds}
\end{align}
for all real $0 \leq s \leq 5$.
\label{thm_big_h}
\end{thm}
\begin{proof}
The proof of \eqref{eq_h_bounds} is straightforward,
elementary, and is based on
\eqref{eq_E_exp} in Section~\ref{sec_elliptic};
it will be omitted.
The correctness of Theorem~\ref{thm_big_h} has also been
validated numerically.
\end{proof}
%
%
\begin{remark}
Numerical experiments by the author indicate that
the relative error of the lower bound in \eqref{eq_h_bounds}
is below 0.07 for all $0 \leq s \leq 5$; moreover, this 
error grows roughly linearly with $s$
to $\approx 0.0085$ for all $0 \leq s \leq 0.1$.
\label{rem_h_bounds}
\end{remark}
In the following theorem, we provide a lower bound on $\chi_n$
for certain values of $n$.
\begin{thm}
Suppose that $\alpha$ is a real number, and that
\begin{align}
0 < \alpha < 5c.
\label{eq_khi_lower_1}
\end{align}
Suppose also that $n\geq 2$ is a positive integer, and that
\begin{align}
n > \frac{2c}{\pi} + \frac{\alpha}{2\pi} \cdot 
   \log\left( \frac{16ec}{\alpha} \right).
\label{eq_khi_lower_2}
\end{align}
Then,
\begin{align}
\chi_n > c^2 + \alpha c.
\label{eq_khi_lower}
\end{align}
\label{thm_khi_lower}
\end{thm}
\begin{proof}
Suppose that the function $H:[0,\infty) \to \Rc$ is defined
via \eqref{eq_big_h_def} in Theorem~\ref{thm_exp_term}.
It was observed in the proof of Theorem~\ref{thm_exp_term}
that $H$ is monotonically increasing. We combine
this observation with \eqref{eq_khi_lower_1}, \eqref{eq_khi_lower_2}
and Theorem~\ref{thm_big_h} to conclude that
\begin{align}
H\left( \frac{\pi n}{2c} - 1 \right) > 
\frac{\alpha}{4c} \cdot \log\left( \frac{16ec}{\alpha} \right) \geq 
 \frac{\alpha}{c}.
\label{eq_khi_lower_a}
\end{align}
Thus \eqref{eq_khi_lower} follows from the combination
of \eqref{eq_khi_lower_a} and Theorem~\ref{thm_exp_term}.
\end{proof}
In the following theorem,
we provide upper and lower bounds on $1-t_n$,
where
$t_n$ is the maximal root of $\psi_n$ in the interval $(-1,1)$,
in terms of $\chi_n-c^2$. This theorem
is illustrated in Figure~\ref{fig:test170}.
\begin{thm}
Suppose that
\begin{align}
c > \frac{10}{\pi}.
\label{eq_tn_simple_1}
\end{align}
Suppose also that $n \geq 2$ is a positive integer, and that
\begin{align}
n > \frac{2c}{\pi} + 1 + \frac{1}{4} \cdot \log(c).
\label{eq_tn_simple_2}
\end{align}
Suppose furthermore that $t_n$ is the maximal root of $\psi_n$
in the interval $(-1,1)$. Then,
\begin{align}
\chi_n > c^2 + \frac{\pi}{2} \cdot c,
\label{eq_tn_simple_khi}
\end{align}
and also,
\begin{align}
\frac{\pi^2}{8 \cdot (1+\sqrt{2})} \cdot \frac{1}{\chi_n-c^2} <
1-t_n <
\frac{2\pi^2}{\chi_n - c^2}.
\label{eq_tn_simple}
\end{align}
\label{thm_tn_simple}
\end{thm}
\begin{proof}
We combine \eqref{eq_tn_simple_1}, \eqref{eq_tn_simple_2} and
Theorem~\ref{thm_khi_lower} to obtain \eqref{eq_tn_simple_khi}.
Then, we combine \eqref{eq_tn_simple_khi} with 
Theorems~\ref{thm_tn_upper},~\ref{thm_tn_lower} to obtain
\eqref{eq_tn_simple}.
\end{proof}

\subsection{Growth Properties of PSWFs}
\label{sec_growth}
In this subsection, we establish several bounds
on $\abrk{\psi_n}$ and $\abrk{\psi_n'}$.
Throughout this subsection $c > 0$ is 
a fixed positive real number.
The principal results of this subsection
are Theorems~\ref{thm_psi1}, \ref{thm_extrema}, \ref{thm_q0}.
The following lemma is a technical tool to be used in
the rest of this subsection.

\begin{lem}
Suppose that $n \geq 0$ is a non-negative integer, and that
the functions $p, q: \Rc \to \Rc$ are defined via 
\eqref{eq_prufer_1} in Section~\ref{sec_prufer}. 
Suppose also that the functions 
$Q, \tilde{Q} : (0,\min\left\{\sqrt{\chi_n}/c,1\right\}) \to \Rc$ are defined,
respectively, via the formulae
\begin{align}
Q(t) 
& \; = \psi_n^2(t) + \frac{ p(t) }{ q(t) } \cdot \brk{\psi'_n(t)}^2 
 = \psi_n^2(t) + 
       \frac{ \brk{1-t^2} \cdot \brk{\psi'_n(t)}^2 }{ \chi_n-c^2 t^2}
\label{eq_psi1_q}
\end{align}
and
\begin{align}
\tilde{Q}(t) 
& \; = p(t) \cdot q(t) \cdot Q(t)  \nonumber \\
& \; = \brk{1 - t^2} \cdot \brk{ \brk{\chi_n - c^2 t^2} \cdot \psi_n^2(t) +
                           \brk{1 - t^2} \cdot \brk{\psi'_n(t)}^2 }.
\label{eq_psi1_qtilde}
\end{align}
Then, $Q$ is increasing in the interval
$\brk{0, \min\cbrk{\sqrt{\chi_n}/c, 1}}$,
and $\tilde{Q}$ is decreasing in the interval 
$\brk{0, \min\cbrk{\sqrt{\chi_n}/c, 1}}$.
\label{lem_Q_Q_tilde}
\end{lem}
\begin{proof}
%
%
We differentiate $Q$, defined via \eqref{eq_psi1_q}, with respect to $t$
to obtain
\begin{align}
Q'(t) = & \; 2 \cdot \psi_n(t) \cdot \psi_n'(t) + \left(
        \frac{2c^2 t \cdot (1-t^2)}{(\chi_n-c^2t^2)^2}-\frac{2t}{\chi_n-c^2t^2}
        \right) \cdot \left(\psi_n'(t)\right)^2 + \nonumber \\
  & \; \frac{2\cdot(1-t^2)}{\chi_n-c^2t^2} \cdot \psi_n'(t) \cdot \psi_n''(t).
\label{eq_dq_long}
\end{align}
Due to \eqref{eq_prolate_ode} in Section~\ref{sec_pswf},
\begin{align}
\psi_n''(t) = \frac{2t}{1-t^2} \cdot \psi_n'(t) - 
\frac{\chi_n-c^2 t^2}{1-t^2} \cdot \psi_n(t),
\label{eq_ddpsi}
\end{align}
for all $-1 < t < 1$. We substitute \eqref{eq_ddpsi} into \eqref{eq_dq_long}
and carry out straightforward algebraic manipulations to obtain
\begin{align}
Q'(t) = \frac{2t}{(\chi_n-c^2 t^2)^2} \cdot \left(\chi_n+c^2-2c^2 t^2\right)
        \cdot \left( \psi_n'(t) \right)^2.
\label{eq_dq_short}
\end{align}
Obviously, for all $0 < t < \min\cbrk{\sqrt{\chi_n}/c, 1}$,
\begin{align}
\chi_n+c^2-2c^2 t^2 > 0.
\label{eq_term_pos}
\end{align}
We combine \eqref{eq_dq_short} with \eqref{eq_term_pos} to conclude that
\begin{align}
Q'(t) > 0,
\label{eq_dq_positive}
\end{align}
for all $0 < t < \min\cbrk{\sqrt{\chi_n}/c, 1}$. Then,
we differentiate $\tilde{Q}$, defined via \eqref{eq_psi1_qtilde},
with respect to $t$ to obtain
\begin{align}
\tilde{Q}'(t) = 
& \; -2t \cdot \left( (\chi_n-c^2 t^2) \cdot \psi_n^2(t) + 
                      (1-t^2) \cdot \left(\psi'_n(t)\right)^2 \right) 
     \nonumber \\
& \; + (1-t^2) \cdot \left( -2c^2 t \cdot \psi_n^2(t) +
                  2 \cdot (\chi_n-c^2 t^2) \cdot \psi_n(t) \cdot \psi_n'(t)
     \right. \nonumber \\
& \; \left. \quad -2t \cdot \left(\psi_n'(t)\right)^2 
                  +2 \cdot (1-t^2) \cdot \psi_n'(t) \cdot \psi_n''(t)
                     \right).
\label{eq_dqtilde_long}
\end{align}
We substitute \eqref{eq_ddpsi} into \eqref{eq_dqtilde_long}
and carry out straightforward algebraic manipulations to obtain
\begin{align}
\tilde{Q}'(t) = -2t \cdot (\chi_n + c^2 - 2 c^2 t^2) \cdot \psi_n^2(t).
\label{eq_dqtilde_short}
\end{align}
We combine \eqref{eq_term_pos} with \eqref{eq_dqtilde_short} to conclude
that
\begin{align}
\tilde{Q}'(t) < 0,
\label{eq_dqtilde_negative}
\end{align}
for all $0 < t < \min\cbrk{\sqrt{\chi_n}/c, 1}$.
We combine \eqref{eq_dq_positive} and \eqref{eq_dqtilde_negative}
to finish the proof.
\end{proof}
In the following theorem, we establish
a lower bound on $\abrk{\psi_n(1)}$.
\begin{thm}[bound on $\abrk{\psi_n(1)}$]
Suppose that $\chi_n > c^2$. Then,
\label{thm_psi1}
\begin{align}
\label{eq_psi1_bound}
\abrk{\psi_n(1)} > \frac{1}{\sqrt{2}}.
\end{align}
\end{thm}
\begin{proof}
Suppose that the function $Q: [-1,1] \to \Rc$ is defined via 
\eqref{eq_psi1_q} in Lemma~\ref{lem_Q_Q_tilde}.
Then, $Q$ is increasing in $\left(0, 1\right)$, and is
continuous in $[-1,1]$ (see Lemma~\ref{lem_Q_Q_tilde}
and Theorem~\ref{thm_prolate_ode} in Section~\ref{sec_pswf}).  
Therefore,
\begin{align}
\psi_n^2(t) < Q(t) \leq Q(1) = \psi_n^2(1), 
\end{align}
for all real $0 \leq t < 1$.
Due to Theorem~\ref{thm_pswf_main} in Section~\ref{sec_pswf},
\begin{align}
\frac{1}{2} = \int_0^1 \psi_n^2(t) \; dt < \int_0^1 \psi_n^2(1) \; dt = 
\psi_n^2(1),
\end{align}
which implies \eqref{eq_psi1_bound}.
\end{proof}
The following theorem describes 
some of the properties of the extrema of $\psi_n$ in $(-1,1)$.
\begin{thm}
Suppose that $n \geq 0$ is a non-negative integer, and that
$x, y$ are two arbitrary extremum points of $\psi_n$ in $(-1,1)$.
If $\abrk{x} < \abrk{y}$, then
\begin{align}
\abrk{\psi_n(x)} < \abrk{\psi_n(y)}.
\label{eq_extremum_general}
\end{align}
If, in addition, $\chi_n > c^2$, then
\begin{align}
\abrk{\psi_n(x)} < \abrk{\psi_n(y)} < \abrk{\psi_n(1)}.
\label{eq_extremum_special}
\end{align}
\label{thm_extrema}
\end{thm}
\begin{proof}
We observe that $\abrk{\psi_n}$ is even in $(-1,1)$, and combine
this observation with the fact that the function $Q:[-1,1] \to \Rc$,
defined via \eqref{eq_psi1_q}, is increasing in $(0,1)$
due to Lemma~\ref{lem_Q_Q_tilde}.
\end{proof}
In the following theorem, we provide an upper bound
on the reciprocal of $|\psi_n|$ (if $n$ is even) or
$|\psi'_n|$ (if $n$ is odd) at zero.
\begin{thm}
Suppose that $\chi_n > c^2$. If $n$ is even, then
\begin{align}
\frac{1}{|\psi_n(0)|} \leq 4 \cdot \sqrt{n \cdot \frac{\chi_n}{c^2}}.
\label{eq_q0_even}
\end{align}
If $n$ is odd, then
\begin{align}
\frac{1}{|\psi_n'(0)|} \leq 4 \cdot \sqrt{\frac{n}{c^2}}.
\label{eq_q0_odd}
\end{align}
\label{thm_q0}
\end{thm}
\begin{proof}
Since $\chi_n > c^2$, the inequality
\begin{align}
\psi_n^2(t) \leq \psi_n^2(1) \leq n + \frac{1}{2},
\end{align}
holds
due to Theorem~\ref{thm_psi1_upper_bound} in Section~\ref{sec_pswf} 
and
Theorem~\ref{thm_extrema} above.
Therefore,
\begin{align}
\int_{1 - 1/8n}^1 \psi_n^2(t) \; dt \leq \frac{1}{8} + \frac{1}{16n} <
\frac{3}{16}.
\end{align}
Combined with the orthonormality of $\psi_n$, this yields the inequality
\begin{align}
\int_0^{1 - 1/8n} \psi_n^2(t) \; dt 
= \int_0^1 \psi_n^2(t) \; dt - \int_{1 - 1/8n}^1 \psi_n^2(t) \; dt
\geq 
\frac{1}{2} - \frac{3}{16} = \frac{5}{16}.
\label{eq_quad_fourth}
\end{align}
Since
\begin{align}
\int \frac{ dx }{ \brk{1 - x^2}^2 } = 
\frac{1}{2} \cdot \frac{ x }{ 1 - x^2 } + 
\frac{1}{4} \log \frac{x+1}{1-x},
\end{align}
it follows that 
\begin{align}
& \int_0^{1-1/8n} \frac{ dx }{ \brk{1 - x^2}^2 } = \nonumber \\
& \frac{1}{2} \cdot \frac{1 - 1/8n}{1 - \brk{1-1/8n}^2} +
  \frac{1}{4} \log \frac{2 - 1/8n}{1/8n} = \nonumber \\
& \frac{1}{2} \cdot \frac{ 8n \brk{8n-1} }{ 16n - 1} +
  \frac{1}{4} \log \brk{16n - 1} \leq \nonumber \\
& 4n + n \leq 5n.
\label{eq_quad_8n}
\end{align}
Suppose that the functions $Q(t), \tilde{Q}(t)$ are 
defined for $-1 \leq t \leq 1$, respectively, 
via the formulae \eqref{eq_psi1_q}, \eqref{eq_psi1_qtilde}
in
Lemma~\ref{lem_Q_Q_tilde} in Section~\ref{sec_growth}.
We apply Lemma~\ref{lem_Q_Q_tilde}
with $t_0 = 0$ and $0 < t \leq 1$ to obtain
\begin{align}
Q(0) \cdot \chi_n
& \; = Q(0) \cdot p(0) \cdot q(0) = \tilde{Q}(0) \nonumber \\
& \; \geq \tilde{Q}(t) = c^2
\sbrk{ \psi_n^2(t) + \frac{ \brk{t^2-1} \brk{\psi_n'(t)}^2 }
                                    { \brk{c^2 \cdot t^2 - \chi_n} } }
   \cdot \brk{1 - t^2} \brk{\chi_n/c^2 - t^2} \nonumber \\
& \; \geq c^2 \psi_n^2(t) \brk{1-t^2} \brk{\chi_n/c^2-t^2}
     \geq c^2 \psi_n^2(t) \brk{1-t^2}^2.
\label{eq_quad_q0_est}
\end{align}
It follows from \eqref{eq_quad_fourth}, 
\eqref{eq_quad_8n} and \eqref{eq_quad_q0_est} that
\begin{align}
& 5n \cdot Q(0) \cdot \frac{\chi_n}{c^2} \geq 
Q(0) \cdot \frac{\chi_n}{c^2} 
\int_0^{1 - 1/8n} \frac{ dx }{ \brk{1-x^2}^2 } \geq 
\int_0^{1 - 1/8n} \psi_n^2(t) \; dt \geq \frac{5}{16},
\end{align}
which, in turn, implies that
\begin{align}
\frac{1}{Q(0)} \leq 16 n \cdot \frac{\chi_n}{c^2}.
\label{eq_quad_32n}
\end{align}
If $n$ is even, then $\psi_n'(0)=0$, also, if $n$ is odd, then $\psi_n(0)=0$.
Combined with \eqref{eq_quad_32n}, this observation
yields both \eqref{eq_q0_even} and \eqref{eq_q0_odd}.
\end{proof}

\section{Numerical Results}
\label{sec_numerical}
In this section,
we illustrate the analysis of Section~\ref{sec_analytical}
via several numerical experiments. All the calculations were
implemented in FORTRAN (the Lahey 95 LINUX version) and were
carried out in
double precision.
The algorithms for the evaluation of PSWFs and
their eigenvalues were based on \cite{RokhlinXiaoProlate}.

We illustrate Lemma~\ref{lem_five} in
Figures~\ref{fig:test75a}, \ref{fig:test75b},
via plotting
$\psi_n$ with $\chi_n < c^2$ and
$\chi_n > c^2$, respectively. The relations
\eqref{eq_khi_small} and \eqref{eq_khi_large} hold for
the functions in Figures~\ref{fig:test75a}, 
\ref{fig:test75b}, respectively.
Theorem~\ref{thm_extrema}
holds in both cases, that is,
the absolute value of local extrema of $\psi_n(t)$ increases
as $t$ grows from $0$ to $1$. On the other hand,
\eqref{eq_extremum_special} holds only for 
the function plotted in Figure~\ref{fig:test75b},
as expected. 
%
%
%
\begin{figure} [htbp]
\begin{center}
\includegraphics[width=12cm, bb=81 227 549 564, clip=true]
{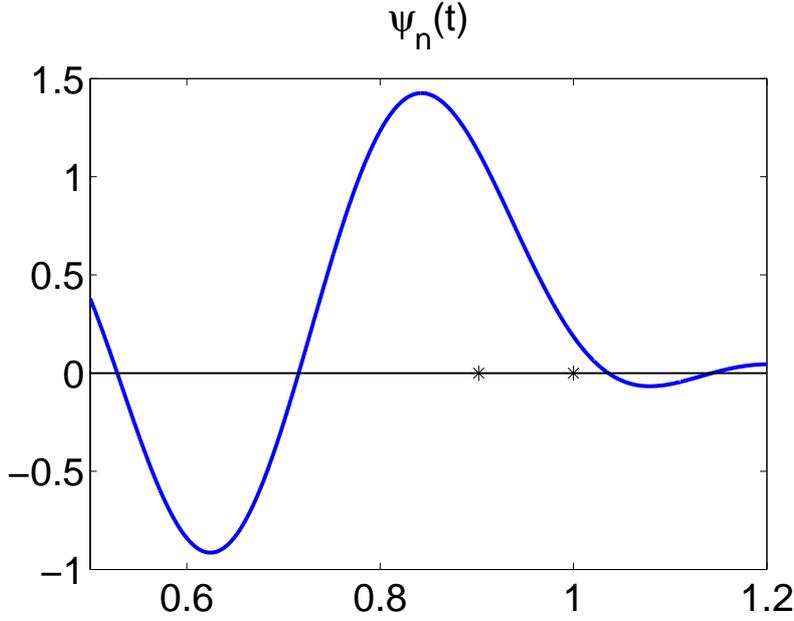}
\caption
{\it
The function $\psi_n(t)$ for $c = 20$ and $n = 9$.
Since $\chi_n \approx 325.42 < c^2$, 
the location of the special points is according to
\eqref{eq_khi_small}
of Lemma~\ref{lem_five}. The points
$\sqrt{\chi_n}/c \approx 0.90197$ and $1$ are marked with asterisks.
Compare to Figure~\ref{fig:test75b}.
}
\label{fig:test75a}
\end{center}
\end{figure}
\begin{figure} [htbp]
\begin{center}
\includegraphics[width=12cm, bb=81 227 529 564, clip=true]
{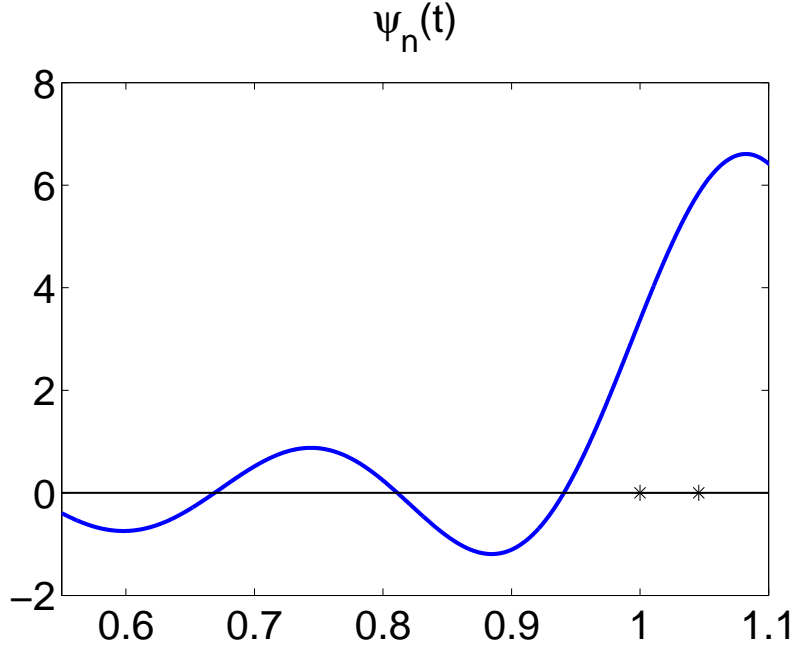}
\caption
{\it
The function $\psi_n(t)$ for $c = 20$ and $n = 14$.
Since $\chi_n \approx 437.36 > c^2$, 
the location of the special points is according to
\eqref{eq_khi_large} of Lemma~\ref{lem_five}. The points $1$ and
$\sqrt{\chi_n}/c \approx 1.0457$ are marked with asterisks. 
Compare to Figure~\ref{fig:test75a}.
}
\label{fig:test75b}
\end{center}
\end{figure}

\begin{table}[htbp]
\begin{center}
\begin{tabular}{c|c|c|c|c|c}
$n$   & 
$\chi_n / c^2$   & 
Above($n$) &
Below($n$) &
$\frac{\text{Above}(n)-n}{n}$ &
$\frac{n-\text{Below}(n)}{n}$ \\[1ex]
\hline
 6
& 0.10104E+01 & 0.65036E+01 & 0.59568E+01 & 0.83927E-01 & 0.71987E-02  \\
 10
& 0.16310E+01 & 0.10498E+02 & 0.99600E+01 & 0.49826E-01 & 0.39974E-02  \\
 15
& 0.29137E+01 & 0.15494E+02 & 0.14963E+02 & 0.32940E-01 & 0.24599E-02  \\
 20
& 0.47078E+01 & 0.20495E+02 & 0.19964E+02 & 0.24737E-01 & 0.17952E-02  \\
 25
& 0.70050E+01 & 0.25496E+02 & 0.24965E+02 & 0.19820E-01 & 0.14066E-02  \\
 30
& 0.98035E+01 & 0.30496E+02 & 0.29965E+02 & 0.16538E-01 & 0.11533E-02  \\
 35
& 0.13103E+02 & 0.35497E+02 & 0.34966E+02 & 0.14189E-01 & 0.97596E-03  \\
 40
& 0.16902E+02 & 0.40497E+02 & 0.39966E+02 & 0.12425E-01 & 0.84521E-03  \\
 45
& 0.21202E+02 & 0.45497E+02 & 0.44966E+02 & 0.11052E-01 & 0.74500E-03  \\
\end{tabular}
\end{center}
\caption{\it
Illustration of Theorems~\ref{thm_n_upper}, \ref{thm_n_lower} with $c = 10$.
The quantities 
Above($n$) and Below($n$) are defined by \eqref{eq_above_below}.
}
\label{t:test77a}
\end{table}
\begin{table}[htbp]
\begin{center}
\begin{tabular}{c|c|c|c|c|c}
$n$   & 
$\chi_n / c^2$   & 
Above($n$) &
Below($n$) &
$\frac{\text{Above}(n)-n}{n}$ &
$\frac{n-\text{Below}(n)}{n}$ \\[1ex]
\hline
 64
& 0.10066E+01 & 0.64590E+02 & 0.63964E+02 & 0.92169E-02 & 0.56216E-03  \\
 70
& 0.10668E+01 & 0.70513E+02 & 0.69971E+02 & 0.73216E-02 & 0.40732E-03  \\
 75
& 0.11290E+01 & 0.75505E+02 & 0.74971E+02 & 0.67341E-02 & 0.38256E-03  \\
 80
& 0.11989E+01 & 0.80502E+02 & 0.79970E+02 & 0.62812E-02 & 0.37011E-03  \\
 85
& 0.12756E+01 & 0.85501E+02 & 0.84970E+02 & 0.58974E-02 & 0.35594E-03  \\
 90
& 0.13584E+01 & 0.90501E+02 & 0.89969E+02 & 0.55623E-02 & 0.34087E-03  \\
 95
& 0.14472E+01 & 0.95500E+02 & 0.94969E+02 & 0.52652E-02 & 0.32589E-03  \\
 100
& 0.15416E+01 & 0.10050E+03 & 0.99969E+02 & 0.49994E-02 & 0.31150E-03  \\
\end{tabular}
\end{center}
\caption{\it
Illustration of Theorems~\ref{thm_n_upper}, \ref{thm_n_lower}
with $c = 100$.
The quantities 
Above($n$) and Below($n$) are defined by \eqref{eq_above_below}.
}
\label{t:test77b}
\end{table}
\begin{table}[htbp]
\begin{center}
\begin{tabular}{c|c|c|c|c|c}
$n$   & 
$\chi_n / c^2$   & 
Above($n$) &
Below($n$) &
$\frac{\text{Above}(n)-n}{n}$ &
$\frac{n-\text{Below}(n)}{n}$ \\[1ex]
\hline
 637
& 0.10005E+01 & 0.63759E+03 & 0.63697E+03 & 0.93059E-03 & 0.51797E-04  \\
 640
& 0.10025E+01 & 0.64055E+03 & 0.63997E+03 & 0.85557E-03 & 0.49251E-04  \\
 645
& 0.10063E+01 & 0.64552E+03 & 0.64497E+03 & 0.80101E-03 & 0.39996E-04  \\
 650
& 0.10105E+01 & 0.65051E+03 & 0.64997E+03 & 0.78412E-03 & 0.39578E-04  \\
 655
& 0.10149E+01 & 0.65551E+03 & 0.65497E+03 & 0.77352E-03 & 0.40527E-04  \\
 660
& 0.10195E+01 & 0.66050E+03 & 0.65997E+03 & 0.76512E-03 & 0.41359E-04  \\
 665
& 0.10243E+01 & 0.66550E+03 & 0.66497E+03 & 0.75777E-03 & 0.41942E-04  \\
 670
& 0.10292E+01 & 0.67050E+03 & 0.66997E+03 & 0.75103E-03 & 0.42321E-04  \\
 675
& 0.10343E+01 & 0.67550E+03 & 0.67497E+03 & 0.74469E-03 & 0.42547E-04  \\
\end{tabular}
\end{center}
\caption{\it
Illustration of Theorems~\ref{thm_n_upper}, \ref{thm_n_lower}
with $c = 1000$.
The quantities 
Above($n$) and Below($n$) are defined by \eqref{eq_above_below}.
}
\label{t:test77c}
\end{table}

In Tables~\ref{t:test77a}, \ref{t:test77b}, \ref{t:test77c}, we
illustrate Theorems~\ref{thm_n_upper}, \ref{thm_n_lower}
in the case of $\chi_n > c^2$.
The band limit $c > 0$ is fixed per table and chosen to 
be equal to 10, 100 and 1000, respectively. 
The first two columns contain
$n$ and the ratio $\chi_n/c^2$. The third and fourth column contain
the upper and lower bound on $n$ defined, respectively,
via \eqref{eq_n_upper_large} in Theorem~\ref{thm_n_upper}
and \eqref{eq_n_lower} in Theorem~\ref{thm_n_lower},
i.e.
\begin{align}
& \text{Below}(n) =
1 + 
\frac{2}{\pi} \int_0^{t_n} \sqrt{ \frac{\chi_n - c^2 t^2}{1 - t^2} } \; dt =
1 + \frac{2}{\pi} \sqrt{\chi_n} \cdot
    E \brk{ \text{asin}\brk{t_n}, \frac{c}{\sqrt{\chi_n}} },
\nonumber \\
& \text{Above}(n) = 
\frac{2}{\pi} \int_0^1 \sqrt{ \frac{\chi_n - c^2 t^2}{1 - t^2} } \; dt =
\frac{2}{\pi} \sqrt{\chi_n} \cdot E \brk{ \frac{c}{\sqrt{\chi_n}} },
\label{eq_above_below}
\end{align}
where $E$ denote the elliptical integrals of Section~\ref{sec_elliptic},
and $t_n$ is the maximal root of $\psi_n$ in $\brk{-1, 1}$
(see also \eqref{eq_jan_n_both_with_e}).
The fifth and sixth columns contain the relative errors of these bounds.
The first row corresponds to the minimal $n$ for which $\chi_n > c^2$.
We observe that, for a fixed $c$, the bounds become more accurate
as $n$ grows. Also, for $n = \lceil 2c/\pi \rceil + 1$
the accuracy improves as $c$ grows. Moreover, 
the lower bound is always more accurate than the upper bound.

\begin{table}[htbp]
\begin{center}
\begin{tabular}{c|c|c|c|c|c}
$n$   & 
$\chi_n / c^2$   & 
Above($n$) &
Below($n$) &
$\frac{\text{Above}(n)-n}{n}$ &
$\frac{n-\text{Below}(n)}{n}$ \\[1ex]
\hline
 1
& 0.29824E-01 & 0.10395E+01 & 0.10000E+01 & 0.39511E-01 & 0.00000E+00  \\
 9
& 0.18531E+00 & 0.90625E+01 & 0.89818E+01 & 0.69444E-02 & 0.20214E-02  \\
 19
& 0.36985E+00 & 0.19069E+02 & 0.18981E+02 & 0.36421E-02 & 0.10180E-02  \\
 29
& 0.54240E+00 & 0.29075E+02 & 0.28980E+02 & 0.25825E-02 & 0.69027E-03  \\
 39
& 0.70125E+00 & 0.39082E+02 & 0.38979E+02 & 0.21102E-02 & 0.53327E-03  \\
 49
& 0.84356E+00 & 0.49096E+02 & 0.48978E+02 & 0.19543E-02 & 0.45122E-03  \\
 54
& 0.90685E+00 & 0.54110E+02 & 0.53977E+02 & 0.20330E-02 & 0.43263E-03  \\
 59
& 0.96278E+00 & 0.59146E+02 & 0.58974E+02 & 0.24725E-02 & 0.44189E-03  \\
 63
& 0.99867E+00 & 0.63420E+02 & 0.62966E+02 & 0.66661E-02 & 0.53355E-03  \\
\end{tabular}
\end{center}
\caption{\it
Illustration of Theorems~\ref{thm_n_upper}, \ref{thm_n_lower} with $c = 100$.
The quantities 
Above($n$) and Below($n$) are defined by \eqref{eq_above_below_small}.
}
\label{t:test99}
\end{table}
In Table~\ref{t:test99}, 
we illustrate Theorems~\ref{thm_n_upper}, \ref{thm_n_lower}
in the case $\chi_n < c^2$ with $c = 100$.
The structure of Table~\ref{t:test99} is the same as that
of Tables~\ref{t:test77a}, \ref{t:test77b}, \ref{t:test77c}
with the only difference: 
the third and fourth column contain
the upper and lower bound on $n$ given, respectively,
via
\eqref{eq_n_upper_small} in Theorems~\ref{thm_n_upper}
and \eqref{eq_n_lower} in Theorem~\ref{thm_n_lower}, i.e.
\begin{align}
& \text{Below}(n) =
1 + 
\frac{2}{\pi} \int_0^{t_n} \sqrt{ \frac{\chi_n - c^2 t^2}{1 - t^2} } \; dt =
1 + \frac{2}{\pi} \sqrt{\chi_n} \cdot
    E \brk{ \text{asin}\brk{t_n}, \frac{c}{\sqrt{\chi_n}} }
\nonumber \\
& \text{Above}(n) = 
\frac{2}{\pi} \int_0^T \sqrt{ \frac{\chi_n - c^2 t^2}{1 - t^2} } \; dt =
\frac{2}{\pi} \sqrt{\chi_n} \cdot 
    E \brk{ \text{asin}\brk{T}, \frac{c}{\sqrt{\chi_n}} },
\label{eq_above_below_small}
\end{align}
where $t_n$ and $T$ are the maximal roots of $\psi_n$ and $\psi_n'$ 
in the interval $\brk{-1, 1}$, respectively.
The values in the first row grow up to $\lfloor 2c/\pi \rfloor$,
in correspondence with Theorem~\ref{thm_n_and_khi} in 
Section~\ref{sec_sharp}. We observe that both bounds in the third
and fourth columns are correct and the lower bound is always more
accurate. This behavior is similar to that observed in
Tables~\ref{t:test77a}, \ref{t:test77b}, \ref{t:test77c}.



\begin{table}[htbp]
\begin{center}
\begin{tabular}{c|c|c|c|c}
$n$   & 
$\brk{n - 2c/\pi - 1}/{c}$ &
$\chi_n$   & 
$\brk{\frac{\pi}{2}\brk{n+1}}^2$ &
$\brk{\frac{\pi}{2}\brk{n+1}}^2 / \chi_n$ - 1 \\[1ex]
\hline
 640
& 0.23802E-02 & 0.10025E+07 & 0.10138E+07 & 0.11248E-01  \\
 660
& 0.22380E-01 & 0.10195E+07 & 0.10781E+07 & 0.57443E-01  \\
 680
& 0.42380E-01 & 0.10395E+07 & 0.11443E+07 & 0.10082E+00  \\
 700
& 0.62380E-01 & 0.10615E+07 & 0.12125E+07 & 0.14229E+00  \\
 720
& 0.82380E-01 & 0.10850E+07 & 0.12827E+07 & 0.18215E+00  \\
 740
& 0.10238E+00 & 0.11100E+07 & 0.13548E+07 & 0.22054E+00  \\
 760
& 0.12238E+00 & 0.11363E+07 & 0.14289E+07 & 0.25757E+00  \\
 780
& 0.14238E+00 & 0.11637E+07 & 0.15050E+07 & 0.29330E+00  \\
 800
& 0.16238E+00 & 0.11923E+07 & 0.15831E+07 & 0.32777E+00  \\
 820
& 0.18238E+00 & 0.12219E+07 & 0.16631E+07 & 0.36105E+00  \\
\end{tabular}
\end{center}
\caption{\it
Illustration of Theorem~\ref{thm_khi_n_square} with $c = 1000$.
}
\label{t:test80a}
\end{table}
\begin{table}[htbp]
\begin{center}
\begin{tabular}{c|c|c|c|c}
$n$   & 
$\brk{n - 2c/\pi - 1}/{c}$ &
$\chi_n$   & 
$\brk{\frac{\pi}{2}\brk{n+1}}^2$ &
$\brk{\frac{\pi}{2}\brk{n+1}}^2 / \chi_n$ - 1 \\[1ex]
\hline
 6400
& 0.32802E-02 & 0.10022E+09 & 0.10110E+09 & 0.87670E-02  \\
 6600
& 0.23280E-01 & 0.10191E+09 & 0.10751E+09 & 0.55007E-01  \\
 6800
& 0.43280E-01 & 0.10390E+09 & 0.11413E+09 & 0.98410E-01  \\
 7000
& 0.63280E-01 & 0.10609E+09 & 0.12094E+09 & 0.13991E+00  \\
 7200
& 0.83280E-01 & 0.10845E+09 & 0.12795E+09 & 0.17979E+00  \\
 7400
& 0.10328E+00 & 0.11094E+09 & 0.13515E+09 & 0.21821E+00  \\
 7600
& 0.12328E+00 & 0.11357E+09 & 0.14255E+09 & 0.25526E+00  \\
 7800
& 0.14328E+00 & 0.11631E+09 & 0.15016E+09 & 0.29102E+00  \\
 8000
& 0.16328E+00 & 0.11916E+09 & 0.15795E+09 & 0.32552E+00  \\
 8200
& 0.18328E+00 & 0.12213E+09 & 0.16595E+09 & 0.35883E+00  \\
\end{tabular}
\end{center}
\caption{\it
Illustration of Theorem~\ref{thm_khi_n_square} with $c = 10000$.
}
\label{t:test80b}
\end{table}
In Tables~\ref{t:test80a}, \ref{t:test80b}, we illustrate
Theorem~\ref{thm_khi_n_square} with $c = 1000$ and $c = 10000$, respectively.
The first column contains the PSWF index $n$, 
which starts from roughly $2c/\pi$ and
increases by steps of $c/50$. The second column displays the normalized
distance $d_n$ between $n$ and  $(2c/\pi+1)$, defined via the formula
\begin{align}
d_n = \frac{n-2c/\pi-1}{c}.
\label{eq_dn}
\end{align}
The third column contains $\chi_n$. The fourth 
and fifth column
contain the upper bound on $\chi_n$, defined in Theorem~\ref{thm_khi_n_square},
and the relative error of this bound, respectively. We observe
that the bound
is slightly better for $c = 10000$,
if we keep $d_n$ fixed. On the other hand, for a fixed $c$,
this bound
deteriorates as $n$ grows. In fact, 
starting from $n \approx (2/\pi + 1/6) \cdot c$, this bound
becomes even worse than \eqref{eq_khi_crude} (this value is $n = 825$ for
$c = 1000$ and $n = 8254$ for $c = 10000$). Since 
Theorem~\ref{thm_khi_n_square} is a simplification of
more accurate Theorems~\ref{thm_n_upper},~\ref{thm_n_lower},
the latter observation
is not surprising. Nevertheless, the high accuracy for $n \approx 2c/\pi$
and the simplicity of the estimate make Theorem~\ref{thm_khi_n_square}
useful (see also Figure~\ref{fig:test171a}).
\begin{figure} [htbp]
\begin{center}
\includegraphics[width=12cm, bb=81 227 529 564, clip=true]
{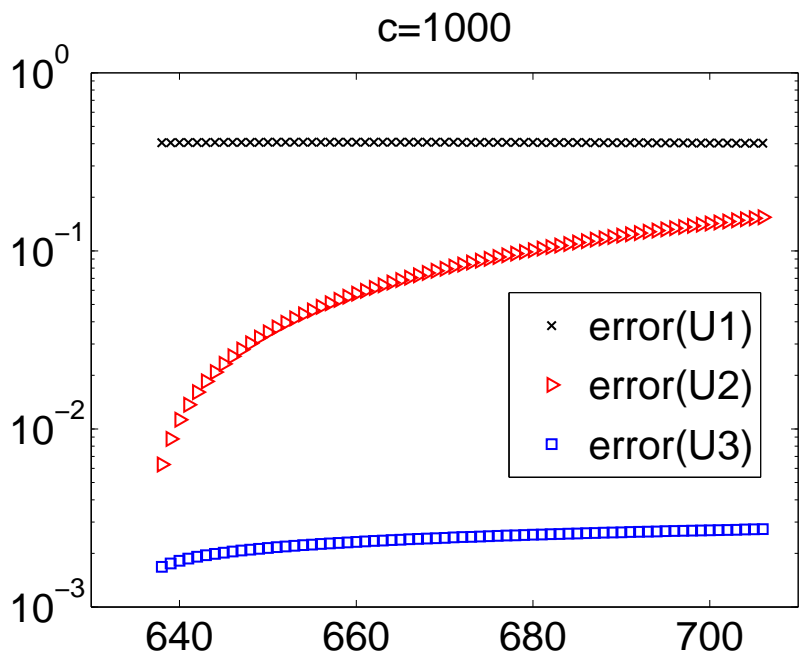}
\caption
{\it
Relative error of upper bounds on $\chi_n$ with $c=1000$,
on the logarithmic scale.
The bounds are defined, respectively, via
\eqref{eq_u1}, \eqref{eq_u2}, \eqref{eq_u3}.
}
\label{fig:test171a}
\end{center}
\end{figure}

In Figure~\ref{fig:test171a}, we illustrate 
Theorems~\ref{thm_khi_crude},~\ref{thm_khi_n_square},~\ref{thm_exp_term}
via comparing the relative accuracy
of the corresponding upper bounds on $\chi_n$.
More specifically, we choose $c=1000$, and, for each integer
$630 \leq n \leq 710$, we evaluate numerically the following quantities.
First, we compute $\chi_n$ (see \eqref{eq_prolate_ode}
in Section~\ref{sec_pswf}). Second, we compute the upper bound
on $\chi_n$, defined via the right-hand side of \eqref{eq_khi_crude}
of Theorem~\ref{thm_khi_crude} in Section~\ref{sec_pswf}, namely,
\begin{align}
U_1(n) = c^2 + n \cdot (n+1).
\label{eq_u1}
\end{align}
Third, we compute the upper bound on $\chi_n$, defined via
\eqref{eq_khi_n_square} of Theorem~\ref{thm_khi_n_square}, namely,
\begin{align}
U_2(n) = \left( \frac{\pi}{2} \cdot (n+1) \right)^2.
\label{eq_u2}
\end{align}
Finally, we compute the upper bound on $\chi_n$, defined via
\eqref{eq_khi_via_h} of Theorem~\ref{thm_exp_term}, namely,
\begin{align}
U_3(n) = c^2 \cdot 
\left(1 + H\left( \frac{\pi n}{2c} - 1 + \frac{3\pi}{2c}\right) \right),
\label{eq_u3}
\end{align}
where $H$ is defined via \eqref{eq_big_h_def} in Theorem~\ref{thm_exp_term}.
In Figure~\ref{fig:test171a}, we plot the relative errors of 
$U_1(n),U_2(n),U_3(n)$ as functions of $n$, on the logarithmic scale.

We observe that $U_1(n)$ significantly overestimates $\chi_n$, and the
relative accuracy of $U_1(n)$ remains roughly the same for all 
$630 \leq n \leq 710$.
On the other hand, the relative accuracy of $U_2(n)$ is higher
than that of $U_1(n)$; however, it deteriorates as
$n$ grows: from below $0.01$ for $n \leq 640$ to above $0.1$
for $n \geq 680$ (see also Table~\ref{t:test80a} above).
Finally, $U_3(n)$ displays much higher relative accuracy than
both $U_1(n)$ and $U_2(n)$: the relative accuracy of $U_3(n)$ remains
below $0.004$ for all $630 \leq n \leq 710$.

\begin{figure} [htbp]
\begin{center}
\includegraphics[width=12cm, bb=81 227 529 564, clip=true]
{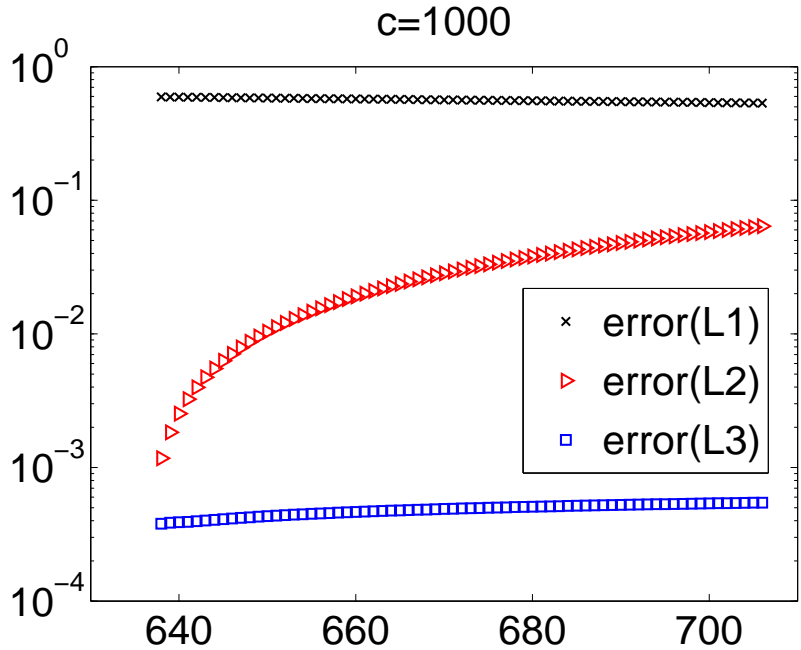}
\caption
{\it
Relative error of lower bounds on $\chi_n$ with $c=1000$,
on the logarithmic scale.
The bounds are defined, respectively, via
\eqref{eq_l1}, \eqref{eq_l2}, \eqref{eq_l3}.
}
\label{fig:test171b}
\end{center}
\end{figure}

In Figure~\ref{fig:test171b}, we illustrate 
Theorems~\ref{thm_khi_crude},~\ref{thm_n_and_khi},~\ref{thm_exp_term}
via comparing the relative accuracy
of the corresponding lower bounds on $\chi_n$.
More specifically, we choose $c=1000$, and, for each integer
$630 \leq n \leq 710$, we evaluate numerically the following quantities.
First, we compute $\chi_n$ (see \eqref{eq_prolate_ode}
in Section~\ref{sec_pswf}). Second, we compute the lower bound
on $\chi_n$, defined via the left-hand side of \eqref{eq_khi_crude}
of Theorem~\ref{thm_khi_crude} in Section~\ref{sec_pswf}, namely,
\begin{align}
L_1(n) = n \cdot (n+1).
\label{eq_l1}
\end{align}
Third, we compute the trivial lower bound on $\chi_n$, established
in Theorem~\ref{thm_n_and_khi}, namely,
\begin{align}
L_2(n) = c^2.
\label{eq_l2}
\end{align}
Finally, we compute the lower bound on $\chi_n$, defined via
\eqref{eq_khi_via_h} of Theorem~\ref{thm_exp_term}, namely,
\begin{align}
L_3(n) = c^2 \cdot 
\left(1 + H\left( \frac{\pi n}{2c} - 1\right) \right),
\label{eq_l3}
\end{align}
where $H$ is defined via \eqref{eq_big_h_def} in Theorem~\ref{thm_exp_term}.
In Figure~\ref{fig:test171b}, we plot the relative errors of 
$L_1(n),L_2(n),L_3(n)$ as functions of $n$, on the logarithmic scale.

We observe that $L_1(n)$ significantly underestimates $\chi_n$, and the
relative accuracy of $L_1(n)$ remains roughly the same for all 
$630 \leq n \leq 710$. Even the trivial lower bound $L_2(n) = c^2$
displays a higher relative accuracy, which, obviously,
deteriorates as $n$ grows.
Finally, $L_3(n)$ is much more accurate than
both $L_1(n)$ and $L_2(n)$: the relative accuracy of $L_3(n)$ remains
below $0.0006$ for all $630 \leq n \leq 710$. We also observe,
that the relative accuracy of $L_3(n)$ is about an order of magnitude
higher than that of $U_3(n)$, defined via \eqref{eq_u3} above
(see Figure~\ref{fig:test171a}).

\begin{table}[htbp]
\begin{center}
\begin{tabular}{c|c|c|c|c|c}
$i$   & 
$t_{i+1} - t_i$ &
$\frac{\pi}{f(t_{i+1}) + v(t_{i+1})/2}$ & 
$\frac{\pi}{f(t_i)}$ &
lower error &
upper error \\[1ex]
\hline
 44
& 0.27468E-01 & 0.27464E-01 & 0.27470E-01 & 0.13152E-03 & 0.63357E-04  \\
 46
& 0.27453E-01 & 0.27439E-01 & 0.27460E-01 & 0.52432E-03 & 0.24265E-03  \\
\hline
 60
& 0.26685E-01 & 0.26573E-01 & 0.26741E-01 & 0.42160E-02 & 0.21008E-02  \\
 62
& 0.26437E-01 & 0.26303E-01 & 0.26506E-01 & 0.50867E-02 & 0.25968E-02  \\
\hline
 70
& 0.24700E-01 & 0.24418E-01 & 0.24863E-01 & 0.11404E-01 & 0.66360E-02  \\
 72
& 0.23948E-01 & 0.23602E-01 & 0.24158E-01 & 0.14473E-01 & 0.87772E-02  \\
\hline
 84
& 0.96757E-02 & 0.81279E-02 & 0.10948E-01 & 0.15996E+00 & 0.13147E+00  \\
 86
& 0.39568E-02 & 0.22125E-02 & 0.55074E-02 & 0.44083E+00 & 0.39188E+00  \\
\end{tabular}
\end{center}
\caption{\it
Illustration of Theorem~\ref{thm_zeros_inside_bounds} 
with $c = 100$ and $n = 87$.
}
\label{t:test98a}
\end{table}
\begin{table}[htbp]
\begin{center}
\begin{tabular}{c|c|c|c|c|c}
$i$   & 
$t_{i+1} - t_i$ &
$\frac{\pi}{f(t_{i+1}) + v(t_{i+1})/2}$ & 
$\frac{\pi}{f(t_i)}$ &
lower error &
upper error \\[1ex]
\hline
 336
& 0.30967E-02 & 0.30967E-02 & 0.30967E-02 & 0.19367E-05 & 0.59233E-06  \\
 338
& 0.30967E-02 & 0.30967E-02 & 0.30967E-02 & 0.52185E-05 & 0.86461E-06  \\
\hline
 400
& 0.30948E-02 & 0.30945E-02 & 0.30949E-02 & 0.11172E-03 & 0.10078E-04  \\
 402
& 0.30947E-02 & 0.30944E-02 & 0.30947E-02 & 0.11547E-03 & 0.10427E-04  \\
\hline
 500
& 0.30813E-02 & 0.30802E-02 & 0.30815E-02 & 0.37302E-03 & 0.41125E-04  \\
 502
& 0.30808E-02 & 0.30797E-02 & 0.30810E-02 & 0.38101E-03 & 0.42311E-04  \\
\hline
 601
& 0.30109E-02 & 0.30065E-02 & 0.30118E-02 & 0.14549E-02 & 0.30734E-03  \\
 603
& 0.30071E-02 & 0.30025E-02 & 0.30080E-02 & 0.15168E-02 & 0.32775E-03  \\
\hline
 667
& 0.10176E-02 & 0.85504E-03 & 0.11505E-02 & 0.15973E+00 & 0.13065E+00  \\
 669
& 0.41703E-03 & 0.23323E-03 & 0.58020E-03 & 0.44073E+00 & 0.39128E+00  \\
\end{tabular}
\end{center}
\caption{\it
Illustration of Theorem~\ref{thm_zeros_inside_bounds} 
with $c = 1000$ and $n = 670$.
}
\label{t:test98b}
\end{table}
\begin{figure} [htbp]
\begin{center}
\includegraphics[width=12cm, bb=81 227 529 564, clip=true]
{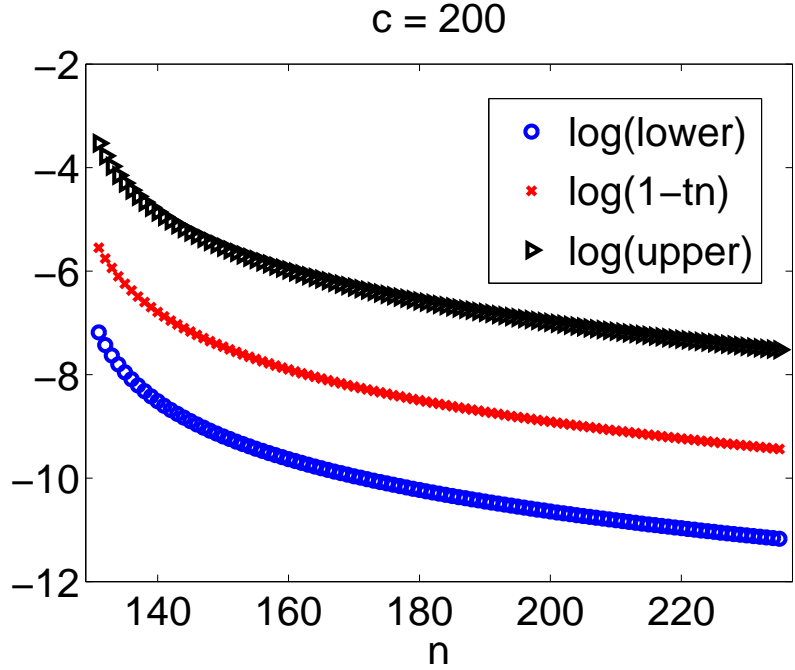}
\caption
{\it
Illustration of Theorem~\ref{thm_tn_simple} with $c=200$.
Here $t_n$ is the maximal root of $\psi_n$
in $(-1,1)$, while the lower 
and upper bounds are defined via \eqref{eq_tn_simple_lower},
\eqref{eq_tn_simple_upper}, respectively.
}
\label{fig:test170}
\end{center}
\end{figure}

In Tables~\ref{t:test98a}, \ref{t:test98b}, we illustrate
Theorems~\ref{thm_zeros_inside_bounds}, \ref{thm_spacing_inside},
with $c = 100, n = 87$ and
$c = 1000, n = 670$, respectively.
The first column contains the index $i$ of the $i$th root $t_i$ of $\psi_n$
inside $\brk{-1, 1}$.
The second column contains the difference between 
two consecutive roots $t_{i+1}$ and $t_i$.
The third and fourth columns contain, respectively,
the lower and upper bounds on this difference, given via
\eqref{eq_zeros_inside_bounds}
in Theorem~\ref{thm_zeros_inside_bounds}. The last two columns
contain the relative errors of these bounds. 
We observe that
both estimates are fairly accurate when $t_i$ is far from 1,
and the accuracy increases with $c$. The best relative accuracy
is about 0.01\% for $c = 100$ and 0.0001\% for $c = 1000$. Both
bounds deteriorate
as $i$ grows to $n$. 
For both values of $c$ the relative accuracy
of the lower bound for $i = n-1$ is as low as 44\%, and that of
the upper bound is about 39\%. 
In general, the upper bound is always more accurate.
We also note that $t_{i+1}-t_i$ decreases
monotonically as $i$ grows, which confirms Theorem~\ref{thm_spacing_inside},
since $\chi_n > c^2$ in both cases.

We illustrate Theorem~\ref{thm_tn_simple} in Figure~\ref{fig:test170}.
We choose $c = 200$, and, for each integer $130 \leq n \leq 230$, we
evaluate numerically the following quantities. First, we compute the maximal
root $t_n$ of $\psi_n$ in $(-1,1)$. Second, we evaluate the
eigenvalue $\chi_n$ (see \eqref{eq_prolate_ode} in Section~\ref{sec_pswf}).
Then, we compute the lower and upper bounds on $1-t_n$, established
in Theorem~\ref{thm_tn_simple}, namely,
\begin{align}
\label{eq_tn_simple_lower}
& lower(n) = \frac{\pi^2}{8 \cdot (1+\sqrt{2})} \cdot \frac{1}{\chi_n-c^2},
\\
& upper(n) = \frac{2\pi^2}{\chi_n-c^2}.
\label{eq_tn_simple_upper}
\end{align}
In Figure~\ref{fig:test170}, we plot $\log(lower(n))$, 
$\log(upper(n))$ and $\log(1-t_n)$,
as functions of $n$.

We observe that neither of \eqref{eq_tn_simple_lower}, 
\eqref{eq_tn_simple_upper}
is a very accurate estimate of $1-t_n$. 
Nevertheless, they correctly capture the behavior
of $1-t_n$, up to a multiplicative constant.
In particular, for all integer $130 \leq n \leq 230$,
\begin{align}
1-t_n = \frac{\xi(n)}{\chi_n-c^2},
\end{align}
where $\xi(n)$ is a real number in the range
\begin{align}
\frac{\pi^2}{8 \cdot (1+\sqrt{2})} < \xi(n) < 2 \pi^2,
\end{align}
as expected
from Theorem~\ref{thm_tn_simple}. In other words, $1-t_n$
is proportional to $(\chi_n-c^2)^{-1}$.





\begin{thebibliography}{99}

\bibitem{Miller}
{\sc Richard K. Miller, Anthony N. Michel},
{\em Ordinary Differential Equations},
Dover Publications, Inc., 1982.

\bibitem{Yoel}
{\sc Yoel Shkolnisky, Mark Tygert, Vladimir Rokhlin},
{\em Approximation of Bandlimited Functions},
Appl. Comput. Harmon. Anal. 21, No. 3, 413-420 (2006).

\bibitem{Glaser}
{\sc Andreas Glaser, Xiangtao Liu, Vladimir Rokhlin},
{\em A fast algorithm for the calculation of the roots of special functions},
SIAM J. Sci. Comput. 29, No. 4, 1420-1438 (2007).

\bibitem{RokhlinXiaoProlate}
{\sc H. Xiao, V. Rokhlin, N. Yarvin},
{\em Prolate spheroidal wavefunctions, quadrature and interpolation},
Inverse Probl. 17, No.4, 805-838 (2001).

\bibitem{RokhlinXiaoApprox}
{\sc Vladimir Rokhlin, Hong Xiao},
{\em Approximate Formulae for Certain Prolate
Spheroidal Wave Functions Valid for Large Value
of Both Order and Band Limit},
Appl. Comput. Harmon. Anal. 22, No. 1, 105-123 (2007).
                              

\bibitem{RokhlinXiaoAsymptotic}
{\sc Hong Xiao, Vladimir Rokhlin},
{\em High-frequency asymptotic expansions for certain prolate spheroidal
wave functions},
J. Fourier Anal. Appl. 9, No. 6, 575-596 (2003).

\bibitem{LandauWidom}
{\sc H. J. Landau, H. Widom},
{\em Eigenvalue distribution of time and frequency limiting},
J. Math. Anal. Appl. 77, 469-481 (1980).

\bibitem{Ryzhik}
{\sc I.S. Gradshteyn,  I.M. Ryzhik},
{\em Table of Integrals, Series, and Products},
Seventh Edition, Elsevier Inc., 2007.

\bibitem{ProlateSlepian1}
{\sc D. Slepian, H. O. Pollak},
{\em Prolate spheroidal wave functions, Fourier analysis,
and uncertainty - I},
Bell System Tech. J. 40, 43-63 (1961).

\bibitem{ProlateLandau1}
{\sc H. J. Landau, H. O. Pollak},
{\em Prolate spheroidal wave functions, Fourier analysis,
and uncertainty - II},
Bell System Tech. J. 40, 65-84 (1961).

\bibitem{ProlateLandau2}
{\sc H. J. Landau, H. O. Pollak},
{\em Prolate spheroidal wave functions, Fourier analysis,
and uncertainty - III:
the dimension of space of essentially time- and band-limited
signals},
Bell System Tech. J. 41, 1295-1336 (1962).
                              
\bibitem{ProlateSlepian2}
{\sc D. Slepian, H. O. Pollak},
{\em Prolate spheroidal wave functions, Fourier analysis,
and uncertainty - IV:
extensions to many dimensions, generalized
prolate spheroidal wave functions},
Bell Syst. Tech. J. November 3009-57 (1964).

\bibitem{ProlateSlepian3}
{\sc D. Slepian},
{\em Prolate spheroidal wave functions, Fourier analysis,
and uncertainty - V: the discrete case},
Bell. System Techn. J. 57, 1371-1430 (1978).


\bibitem{SlepianComments}
{\sc D. Slepian},
{\em Some comments on Fourier analysis, uncertainty, and modeling},
SIAM Rev. 25, 379-393 (1983).
                              
\bibitem{SlepianAsymptotic}
{\sc D. Slepian},
{\em Some asymptotic expansions for prolate spheroidal wave
functions},
J. Math. Phys. 44 99-140 (1965).

\bibitem{PhysicsMorse}
{\sc P. M. Morse, H. Feshbach},
{\em Methods of Theoretical Physics},
New York McGraw-Hill, 1953.

\bibitem{Fuchs}
{\sc W. H. J. Fuchs},
{\em On the eigenvalues of an integral equation
arising in the theory of band-limited signals},
J. Math. Anal. Appl. 9 317-330 (1964).

\bibitem{Grunbaum}
{\sc F. A. Gr\"unbaum, L. Longhi, M. Perlstadt},
{\em Differential operators commuting with finite
convolution integral operators: some
non-Abelian examples},
SIAM J. Appl. Math. 42, 941-955 (1982).

\bibitem{Flammer}
{\sc C. Flammer},
{\em Spheroidal Wave Functions},
Stanford, CA: Stanford University Press, 1956.

\bibitem{Papoulis}
{\sc A. Papoulis},
{\em Signal Analysis},
Mc-Graw Hill, Inc., 1977.

\bibitem{Abramovitz}
{\sc M. Abramowitz, I. A. Stegun},
{\em Handbook of Mathematical Functions with
Formulas, Graphs and Mathematical Tables},
Dover Publications, 1964.

\bibitem{Fedoryuk}
{\sc M.V. Fedoryuk}, 
{\em Asymptotic Analysis of Linear Ordinary Differential Equations}, 
Springer-Verlag, Berlin (1993).

\end{thebibliography}
\end{document}